\documentclass[11pt,a4paper]{article}

\usepackage{abstract}
\usepackage{amsmath}
\usepackage{amssymb}
\usepackage{amsthm}
\usepackage{hyperref}
\usepackage{booktabs} % For pretty tables
\usepackage{caption} % For caption spacing
\usepackage[pagewise]{lineno}
\usepackage[a4paper,
  left=3.5cm,
  right=3.5cm,
  top=3cm,
  bottom=3cm
]{geometry}

\usepackage{subcaption} % For sub-figures
\usepackage{graphicx}
\usepackage{pgfplots}
\usepackage{adjustbox} % Add in the preamble
\usepackage[ruled,vlined]{algorithm2e}
\usepackage{authblk}

\usepackage[all]{nowidow}
\usepackage[utf8]{inputenc}
\usepackage{tikz}
\usetikzlibrary{er,positioning,bayesnet}
\usepackage{multicol}
\definecolor{blue}{HTML}{1F77B4}
\definecolor{orange}{HTML}{FF7F0E}
\definecolor{green}{HTML}{2CA02C}

\pgfplotsset{compat=1.14}

\newcommand{\g}{{\, |\,}}
\newcommand{\Pp}{{\mathbb P}}
\newcommand{\wt}{{\widetilde t}}
\newcommand{\kauno}{{\mathcal L}}
\newcommand{\Nn}{{\mathcal N}}
\newcommand{\Tt}{{\Gamma}}
\newcommand{\oT}{\boldsymbol{T}}

\newcommand{\N}{\mathbb{N}}

\newcommand{\R}{\mathbb{R}}
\newcommand{\Z}{\mathbb{Z}}

\DeclareMathOperator*{\argmax}{argmax}
\DeclareMathOperator*{\argmin}{argmin}
\newtheorem{definition}{Definition}
\newtheorem{proposition}{Proposition}

\newtheorem{remark}{Remark}

\begin{document}
\title{Random tree Besov priors: Data-driven regularisation parameter selection}

\author[1]{Hanne Kekkonen}
\author[2]{Andreas Tataris}

\affil[1]{Delft Institute of Applied Mathematics\\
  Delft University of Technology\\
  The Netherlands\\
  \texttt{h.n.kekkonen@tudelft.nl}}

\affil[2]{Inria, ENSTA Paris, UMA\\
  Institut Polytechnique de Paris\\
   France\\
   \texttt{andreas.tataris@inria.fr}}

\maketitle              % typeset the header of the contribution
%
%\begin{abstract}
%Notes on the EFWI and the Marchenko %integral equations.
%\keywords{}
%\end{abstract}
%
%% 
%$$
\begin{abstract}

We develop a data-driven algorithm for automatically selecting the wavelet density parameters in Bayesian inversion under random tree Besov priors.
One of the key challenges in Bayesian inversion is the construction of priors that are both expressive and computationally feasible. Random tree Besov priors, introduced in \cite{Kekkonen2023}, provide a flexible framework for capturing local regularity properties and sparsity patterns in a wavelet basis. In this paper, we extend this approach by introducing a hierarchical model that enables data-driven selection of the wavelet density parameter, allowing the regularisation strength to adapt across scales while retaining computational efficiency. We focus on nonparametric regression and also present preliminary plug-and-play results for a deconvolution problem.\\


\end{abstract}
\section{Introduction}

Inverse problems arise from the need to estimate unknown quantities from indirect, often incomplete, noisy measurements. They are encountered in various applications, ranging from medical imaging and remote sensing to geophysical exploration and financial modelling. Inverse problems are typically ill-posed, meaning that small perturbations in data or inaccuracies in the model can result in large errors or instabilities in the reconstructions. To obtain stable reconstructions, we must carefully combine prior knowledge about the unknown with the measurements.

In practice, measurements are always finite-dimensional and contaminated by noise. To avoid undesirable artefacts and potential changes in the qualitative behaviour of reconstructions due to discretisation, it is advisable to construct a continuous model that is discretised as late as possible in the analysis and solution process \cite{Haario2004, Lassas2004, Stuart2010}. This leads to a continuous model of the form
\begin{equation*}\label{eq:IP_continuous}
M = Af + \varepsilon \mathbb{W},
\end{equation*}
where $A:X\to Y$ is a linear operator between Banach spaces $X$ and $Y$, $\varepsilon>0$ describes the noise amplitude, and $\mathbb{W}$ is a Gaussian white noise process indexed by $Y$ \cite{Reiss2008}. 

Bayesian inversion provides a powerful framework for addressing inverse problems by combining the indirect, noisy data with prior information in a probabilistic manner. This offers an attractive alternative to regularisation by also delivering information on how uncertainties in the data or forward model affect the solution.
The choice of the prior distribution, which encodes assumptions about the unknown quantity, is central to the Bayesian approach. While Gaussian priors offer computational efficiency and theoretical simplicity, they impose global smoothness constraints, making them less effective for problems such as image and signal reconstruction, where capturing sharp edges and interfaces is essential.

Recently, Besov priors have emerged as a compelling alternative, bridging the gap between rigorous mathematical formulation and practical performance \cite{Lassas2009, Dashti2012}. Besov priors naturally capture local smoothness properties while allowing discontinuities, making them well-suited for problems involving spatially varying regularity \cite{Bui2015, Cui2021, Hamalainen2013, Horst2025, Kolehmainen2012, Rantala2006, Suuronen2020, Wang2017}. Their practical ability to localise and preserve jumps, however, depends strongly on the chosen wavelet basis. Numerical studies have shown that Besov priors based on smoother Daubechies wavelets may fail to preserve discontinuities that are recovered with Haar wavelets \cite{Arridge2019, Kolehmainen2012, Horst2025}. Haar wavelets in turn favour discontinuities aligned with the underlying dyadic grid \cite{Suuronen2020}.

Standard Besov priors treat the wavelet coefficients independently and therefore do not encode the persistence of local features across resolution levels. Random tree Besov priors were introduced to incorporate additional cross-scale structure into the wavelet representation \cite{Kekkonen2023}. In this construction the wavelet coefficients are supported on structured trees, which promotes sparsity and local adaptivity through hierarchical constraints. A local singularity such as a jump typically gives rise to significant wavelet coefficients at related spatial locations over several consecutive scales. Noise can also produce individual large coefficients, but these are less likely to form a coherent spatially aligned pattern across scales. By requiring selected fine-scale coefficients to be connected to selected coefficients at coarser levels, the tree prior favours persistent structures across scales. Random tree Besov priors have also been applied to elliptic PDEs \cite{Schwab2024, Stein2023}.

In \cite{Kekkonen2023} the wavelet density index $\beta$ was selected manually. The wavelet density index controls the sparsity of the wavelet coefficients. Smaller values produce sparser trees, while larger values allow branches to persist to finer resolution levels, leading to functions with more fine-scale structure and local detail. Although \cite{Kekkonen2023} yielded promising results, the optimal choice of $\beta$ remained challenging and data-dependent. Furthermore, employing a single fixed regularisation parameter across all wavelet decomposition levels is overly restrictive, since wavelet coefficients typically exhibit varying degrees of sparsity at different resolutions with fine scales often having sparser representations than coarser scales \cite{Adcock2017}.

In this paper, we address these limitations by presenting an automatic data-driven method for selecting the wavelet density index $\beta$. We introduce a hierarchical Bayesian model in which the wavelet density parameter is allowed to vary across decomposition levels, resulting in a vector-valued parameter $\boldsymbol{\beta}= (\beta_j)_j$. We place a level-dependent hyperprior on $\boldsymbol{\beta}$  and propose a computationally efficient method for estimating $\boldsymbol{\beta}$ directly from data through a sequence of levelwise maximum a posteriori (MAP) estimates. This strategy allows the regularisation strength to adapt flexibly to the scale-specific sparsity inherent in wavelet decompositions and removes manual tuning of $\beta$. We introduce an efficient bottom-up algorithm for the Gaussian and Laplace base priors in which the regularisation parameters are selected level by level and the wavelet coefficient tree is constructed recursively. We provide numerical experiments comparing our automatically chosen $\boldsymbol{\beta}$ with an oracle-tuned fixed scalar $\beta$. 

Random tree Besov priors are conceptually related to other tree-structured Bayesian models. In finite-dimensional signal and image processing, wavelet-domain hidden Markov tree models \cite{Crouse1998,Romberg2001} represent wavelet coefficients via Gaussian mixtures with hidden states forming a Markov tree, capturing clustering of large and small coefficients and cross-scale persistence. In nonparametric density estimation tree-based priors such as Pólya trees \cite{Castillo2017, Castillo2021, Castillo2022}, as well as Bayesian CART-type priors with multiscale shrinkage \cite{Castillo2019}, also exploit hierarchical tree structures. In contrast, random tree Besov priors are constructed directly as Besov function-space priors in a wavelet basis and are tailored for Bayesian inverse problems.

The remainder of this paper is organised as follows. In Section \ref{Sec:Preliminaries}, we review relevant concepts for Besov spaces and wavelet-based random variables, and define the random tree Besov priors in detail. Section \ref{Sec:BetaSelection} describes our proposed method for automatically selecting the level-dependent regularisation parameter $\boldsymbol{\beta}$. Numerical examples illustrating the performance and practical advantages of our approach are provided in Section \ref{Sec:Numerical}.

% contributions
\section{Random tree Besov priors\label{Sec:Preliminaries}}
%\subsection{Bayesian Inversion}

We briefly recall the essential theory of Besov spaces and the construction of Besov measures in the next subsection, before moving to the hierarchical random tree structure.

\subsection{Besov random variables}\label{Sec:BesovMeasures}
In this paper we use the wavelet representation of Besov spaces. For any integer $r \geq 1$ there exist a compactly supported $C^r(\R^d)$ scaling function $\phi$ and wavelets $\Psi^{\ell}$, $\ell=1,\ldots,2^d-1$, which generate an $r$-regular orthonormal wavelet basis of $L^2(\R^d)$ suitable for multi-resolution analysis of smoothness $C^r$ \cite{Daubechies1992}. 
More specifically, if we denote
\begin{align*}
\psi^\ell_{jk}(x)&=2^{jd/2}\Psi^\ell(2^jx-k),\qquad k\in\Z^d,\;  
j\geq 0,  \; \ell\in\kauno:=\{1,\ldots, 2^d-1\},\\
\phi_{k}(x)&= \phi(x-k), \qquad k\in \Z^d.
\end{align*}
then any function $f\in L^2(\R^d)$ can be decomposed in terms of the wavelet basis as
\begin{align*}
f(x)&=\sum_{k\in\Z^d}\langle f, \phi_k\rangle\phi_k(x)
+\sum_{\substack{j\geq 0,\ k\in \Z^d\\ \ell\in\kauno}}\langle f, \psi^\ell_{ jk}\rangle\psi^\ell_{ jk}(x)\\
&=\sum_{k\in\Z^d}f_{-1k}\phi_k(x)+\sum_{\substack{j\geq 0,\ k\in \Z^d
\\ \ell\in\kauno}}f^\ell_{ jk}\psi^\ell_{ jk}(x).
\end{align*}

The Besov space $B_{pp}^s(\R^d)$, with $s<r$, is defined as
\begin{align*}
    B_{pp}^s(\R^d):= \{f\in L^2(\R^d): \|f\|_{B^s_{pp}(\R^d)}<\infty \},
\end{align*}
with the norm
\begin{align*}
\|f\|_{B^s_{pp}(\R^d)}
= \left(\sum_{{j=-1}}^\infty 2^{jp\big(s+\frac d2-\frac dp\big)}
\|\textbf{f}_j\|_{\ell^p}^p\right)^{1/p},
\end{align*}
where
\begin{align*}
\textbf{f}_j&=(f^\ell_{ jk})_{k\in \Z^d, \; \ell\in\kauno}\quad \text{with}\quad
\|\textbf{f}_j\|_{\ell^p}=\left(\sum_{k\in \Z^d, \; \ell\in\kauno}
|f^\ell_{ jk}|^p\right)^{1/p}  \qquad \textrm {if}\;\, j\geq 0,\quad \textrm{and}\\
\textbf{f}_{-1}&=(f_{-1k})_{k\in \Z^d} \quad \text{with}\quad
\|\textbf{f}_{-1}\|_{\ell^p}=\left(\sum_{k\in \Z^d}|f_{-1k}|^p\right)^{1/p}.
\end{align*}
For more details on Besov spaces we refer to 
\cite{Cohen2003,Gine2015,Tartar2007}.

We note that Besov spaces are suitable for representing functions with local smoothness variations, including sharp edges. For example, the Besov space $B^1_{11}$ is close to the space of total variations in the sense that 
\begin{align*}
    B^{1}_{11}(\R^d)\subset W^{1,1}_{loc}(\R^d)\subset B^{1-\varepsilon}_{11}(\R^d), \qquad \textrm{for all}\quad \varepsilon >0.
\end{align*}

For simplicity, we ignore boundary effects below by considering functions with support in the interior of the unit cube $D=[0,1]^d$. 
We could also adjust the definition of the Besov priors for different boundary conditions, and study the behaviour  near the boundary, but since this is application-specific we will not consider it in this paper.

Let $(X^\ell_{jk})_{\ell\in\kauno,\, j\in\N,\,k\in\Z^d}$ be an i.i.d. sequence of real random variables with probability density function $\pi(x)\propto\exp (-\frac{1}{2\kappa^p}|x|^{p})$, $1\leq p<\infty$, denoted by $X^\ell_{jk}\sim \Nn_p(0,\kappa^p)$.
Define the random function
\begin{align}\label{eq:BesovPrior}
f(x)=\sum_{\substack{\ell\in\kauno,\, j\in\N\\ k\in\Z^d}}
2^{-j\big(s+\frac d2-\frac dp\big)} X^\ell_{ j,k}\psi^\ell_{j,k}(x),\quad x\in D,
\end{align}
with $s<r$.  We call $f$ a $B^{s}_{pp}$-random variable.

These random variables take values almost surely in Besov spaces $B^{t}_{pp}$ with $t<s-\frac{d}{p}$, while a realisation $f$ belongs to $B^{s-d/p}_{pp}$ with probability zero \cite{Dashti2012}. Informally, $f$ has density proportional to $\exp(-\frac{1}{2\kappa^p}\|f\|_{B^s_{pp}}^p)$, see \cite{Dashti2012,Lassas2009}. 

If $p=2$, \eqref{eq:BesovPrior} defines a Gaussian measure with Cameron-Martin space $B^s_{22}=H^s$. For Gaussian measures, the Cameron–Martin space characterises admissible translations under which the measure remains equivalent, and governs small-ball probabilities, concentration, and posterior contraction. In contrast, Besov priors require a more delicate analysis. For $p\in[1,2]$, the Hilbert space $B^s_{22}$ describes admissible translations, while the Banach space $B^s_{pp}$ determines concentration and contraction rates \cite{Agapiou2021}. 

Since the Besov priors in \eqref{eq:BesovPrior} are constructed from independent wavelet coefficients, they do not encode dependence between coefficients at neighbouring locations or across resolution levels. Random tree Besov priors introduce hierarchical wavelet coefficient selection to model such cross-scale structure while retaining the function-space construction.\\

\subsection{Random tree Besov variables}\label{Sec:RandomTreeBesovVariables}

We will now introduce the random tree Besov priors.
Since we consider functions in the interior of $D=[0,1]^d$ we see that
$f^\ell_{ jk}=0$ if $k_l<0$ or $k_l>2^j-1$ for some $l=1,\dots,d$. 
We call the remaining set of indices of the wavelet coefficients the {\it entire tree} and denote this by
\begin{align*}
\oT=\{(j,k)\in \N\times \N^d\,,\, j\geq 0,\ k=(k_1,\dots,k_d),\
0\leq k_l\leq 2^j-1\}.
\end{align*}
Here $j$ denotes the resolution level, $k$ the spatial location index, and $\ell$ the number of parallel trees. For 1D signals $\ell=1$ and we can organise the wavelet coefficients into a binary tree. For 2D images $\ell\in\{1,2,3\}$ making the structure more complex since we need to consider three parallel quadtrees: one for vertical, horizontal, and diagonal subbands. 

The idea of the random tree Besov priors presented in \cite{Kekkonen2023} was to use a wavelet expansion to create priors similar to \eqref{eq:BesovPrior} but to choose non-zero wavelets in the sum in a systematic way. To achieve this we will introduce the set of {\it proper subtrees}:
\begin{align*}
\Gamma=\{T\subset\oT\,|\,
\hbox{if $(j,k)\in T$ then $(j-1,[k/2])\in T$}\},
\end{align*}
where $[k/2]=([k_1/2],\dots,[k_d/2])$ is the vector whose
elements are the integer parts of $k_l/2$. This means that for any $T\in\Gamma$ and any node $(j,k)\in T$, all of its ancestors also belong to the tree $T$, and therefore every node in a proper subtree is connected to the root node. 
A random subtree can be generated by a Galton-Watson-type branching mechanism on the dyadic tree, where each potential child node is included independently with probability $\beta$ conditional on its parent being included.

For simplicity,  we assume below that all the bands $\ell$ have the same subtree structure. This assumption can easily be relaxed by considering separate subtrees $T^\ell$ instead.  
We can now introduce the random tree Besov priors.

\begin{definition}\label{Def:TreeRandomVariable}
Let $\{\psi^\ell_{jk}\}_{\ell\in\kauno, j\in\N,\,k\in\Z^d}$ be an $r$-regular orthonormal wavelet basis for $L^2(\R^d)$, and $\beta=2^{\gamma-d}$ with some $\gamma\in(-\infty,d\,]$. Consider pairs $(X,T)$ where $X$ is an $\R^{\oT^{2^d-1}}$-valued random variable, and $T\in {\Tt}$ is a random tree. We assume that $X$ and $T$ are independent random variables, having the following distributions
\begin{itemize}
\item The sequence $X$ consists of i.i.d. random variables  with probability density proportional to $\exp (-\frac{1}{2\kappa^p}|x|^{p})$, $\kappa>0$ and $1\leq p <\infty$, that is, $X_{jk}^\ell\sim\Nn_p(0,\kappa^p)$.
\item The proper subtree $T$ is built recursively. The root node is always included. For each other node, conditional on its parent being in $T$, the node is included independently with probability $\beta$; if its parent is not in $T$, the node is excluded.
\end{itemize}
Let $f$ be the random function
\begin{align*}
f(x)=\sum_{\substack{(j,k)\in T\\ \ell\in\kauno}}
2^{-j\big(s+\frac d2-\frac dp\big)}X^\ell_{j,k}\psi^\ell_{j,k}(x),\quad x\in D,
\end{align*}
where $s<r$. We say that $f$ is a $B^{s}_{pp}$-random variable with wavelet density index $\beta$.
\end{definition}

Figure \ref{RandomDraws} shows three random tree Besov prior draws with wavelet density indices 
$\beta=0.5,0.7,0.9$. For each value of $\beta$,
the same random tree and the same Gaussian coefficients
$X_{j,k}\sim N(0,1)$ are used with the Haar and Daubechies 4 wavelet bases. The first two rows illustrate the effect of the wavelet basis for a fixed tree and coefficient realisation.

In the infinite-dimensional random tree Besov prior, $\beta$ is a branching probability controlling the persistence of the tree to arbitrarily fine scales. Since each node has $2^d$ potential children, the corresponding Galton--Watson process becomes
extinct almost surely when $\beta\leq 2^{-d}$. For $\beta>2^{-d}$ there is a positive probability that the tree persists to arbitrarily fine scales, and conditional on this event the prior can exhibit persistent singular structure \cite{Kekkonen2023}. Larger values of $\beta$ produce denser trees whose branches are more likely to extend to fine resolution levels, as illustrated in Figure \ref{RandomDraws}.

\begin{figure}[t]
\centering
\includegraphics[width=\textwidth]{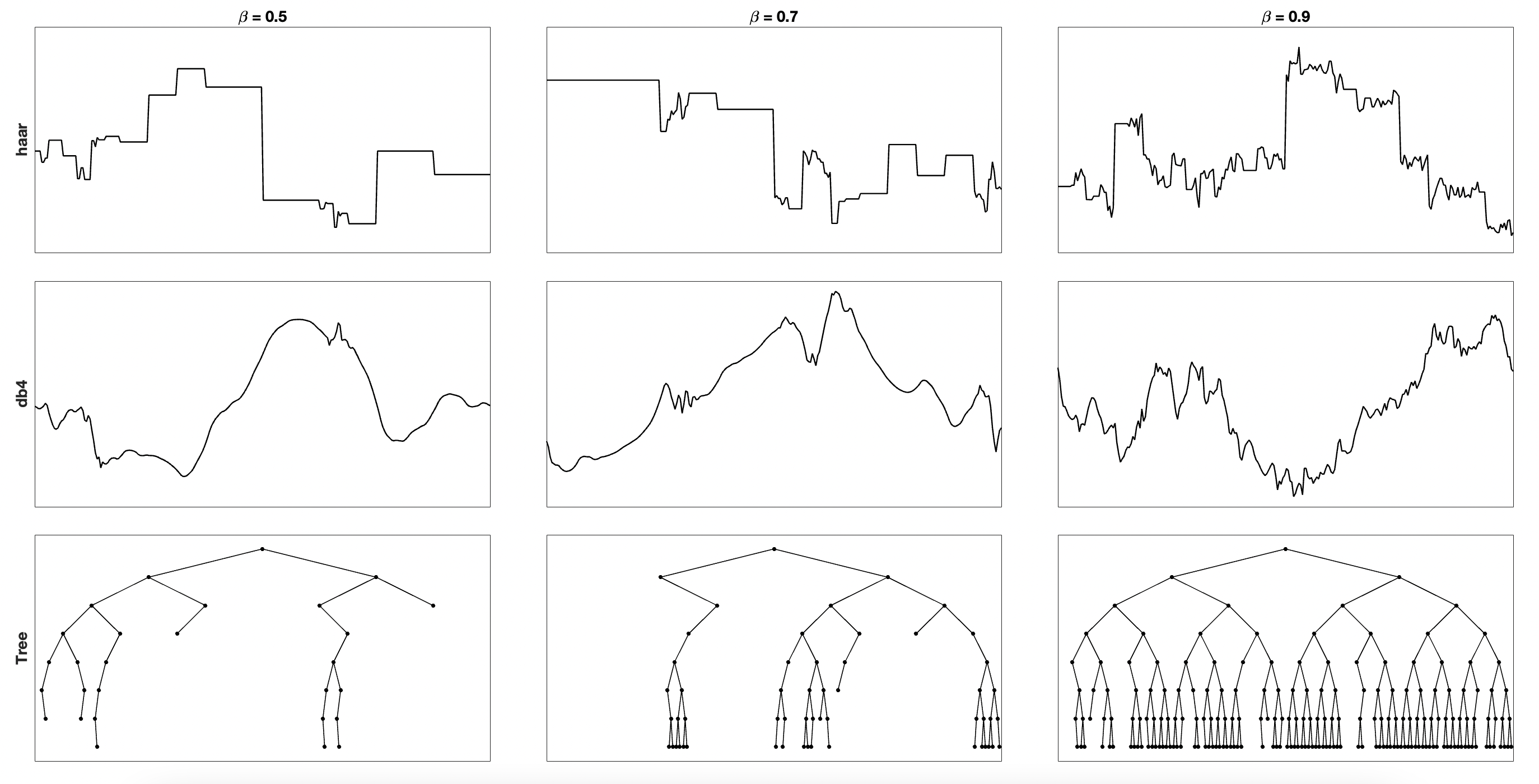}
    \caption{Random tree Besov prior draws with wavelet density indices $\beta=0.5,0.7,0.9$. 
    Top row: reconstructions in the Haar basis.
    Middle row: reconstructions using the same random trees and Gaussian coefficient realisations in the db4 basis. Bottom row: the corresponding random trees.}
    \label{RandomDraws}
\end{figure}

In \cite{Kekkonen2023} the wavelet density index was chosen manually but we would like to choose $\beta$ automatically from data.
Below we will consider level dependent $\boldsymbol{\beta}=(\beta_j)_j$ to allow varying sparsity levels since it is common that the wavelet coefficients of a signal or image are sparser at the finer resolution scales than at the coarser scales. 
We will assign a level-dependent hyperprior for $\boldsymbol{\beta}$ and consider the levelwise maximum a posteriori (MAP) estimates for the wavelet density index.

\section{Tree pruning with automatic $\boldsymbol{\beta}$ selection \label{Sec:BetaSelection}}

In this section, we develop our tree-pruning scheme with automatic selection of the level-dependent wavelet density index $\boldsymbol{\beta}=(\beta_j)_{j=1}^J$. We begin by introducing the problem and recall the tree-pruning algorithm for a fixed, known $\beta$ in Section \ref{Sec:PruningFixedBeta} in order to set notation and highlight the role of $\beta$ as a regularisation parameter. Sections \ref{Sec:PruningGaussian} and \ref{Sec:PruningLaplace} then describe our main contribution: a hierarchical Bayesian model in which each $\beta_j$ is equipped with a hyperprior and is estimated from the data via levelwise MAP, for the Gaussian and Laplace base priors respectively. 

We consider the nonparametric regression problem where
we have observations at $2^{d(J+1)}$, $J\in\N$, regularly spaced points $x_i$ on $D=[0,1]^d$, $d=1,2$, of some unknown function $f$ subject to noise
\begin{align}\label{eq:DenoisingMain}
    M_i=f(x_i)+W_i,
\end{align}
where $W_i\sim \Nn(0,\sigma^2)$ is white noise independent of $f$. We address the cases $p=1$ and $p=2$ which correspond to the non-zero wavelet coefficients of $f$ being drawn from the Laplace and Gaussian distributions respectively. 

The wavelet coefficients can be arranged in a finite tree of the form
\begin{equation*}
\oT=\{(j,k)\in \N\times \N^d\ |\ 0\le j\le J,\ k=(k_1,\dots,k_d),\ 0\leq k_l\leq 2^j-1\}
\end{equation*}
and the non-zero wavelet coefficients will be in a finite proper subtree
\begin{align*}
\Gamma=\{T\subset\oT\,|\,
\hbox{if $(j,k)\in T$ then $(j-1,[k/2])\in T$}\},
\end{align*}
where $[k/2]$ is the integer part of $k/2$. 
We denote $(j',k')\prec(j,k)$ if $(j,k)$ is an ancestor of $(j',k')$.
The complete subtree with a root node $(j,k)$ can then be written as
\begin{align*}
\oT_{(j,k)} =\{(j',k')\in \oT\ |\ (j',k')\preceq (j,k)\}.
\end{align*}

Let $n_j=2^{jd}$ denote the number of potential nodes on level $j$ of one tree. We model the unknown wavelet density indices $\beta_j$, $j=1,\dots,J$, independently across levels with the level-dependent hyperpriors
\begin{align*}
\pi_j(\beta_j)=2^{1+a n_j}(1+a n_j)\Big(\frac{1}{2}-\beta_j\Big)^{a n_j},
\qquad \beta_j\in[0,0.5],\quad a\geq0.
\end{align*}
The exponent is scaled by $n_j$ so that the contribution of the hyperprior has the same order in $n_j$ as the other level-wise terms in the posterior.
The number of nodes $n_j$ increases with the size of the observed signal or image coefficients. The scaling is similar to sample size or discretisation dependent prior scaling often required in nonparametric statistical inverse problems.

Equivalently, up to a normalising constant independent of $\beta_j$, the hyperprior can be factorised as
\begin{align*}
\pi_j(\beta_j)\propto \prod_{k=1}^{n_j}\Big(\frac{1}{2}-\beta_j\Big)^a.
\end{align*}
This factorisation provides a useful node-wise interpretation of the regularisation. This construction could in future be relaxed by allowing location-dependent parameters $\beta_{jk}$, which would provide a way to model spatially varying sparsity and local regularity. We do not pursue such an extension in this paper.

We restrict $\beta_j\in [0,0.5]$ so that the prior acts as a sparsity-promoting regulariser. For $\beta_j<0.5$ we have $-\log(\beta_j)>-\log(1-\beta_j)$, so adding nodes is discouraged \textit{a priori}. The strength of the hyperprior is controlled by the parameter $a$. As $a$ increases, the prior penalises large $\beta_j$ more heavily, leading to sparser trees and consequently stronger regularisation. 

To generate a random subtree with wavelet density $\beta=(\beta_j)_{j=1}^J$, we first consider the nodes independently and set the probability for a node to be included as $\Pp(t_{jk}=1 \g \beta_j)=\beta_j$, and correspondingly the probability for exclusion as $\Pp(t_{jk}=0 \g \beta_j)=1-\beta_j$. 
The resulting proper subtree $T$ then consists of all nodes that connect to the root node, that is, the nodes $(j,k)$ for which
\begin{align*}
\wt_{jk}=\prod_{(j',k')\succeq(j,k)}t_{j'k'} =1.
\end{align*}

We model the unknown function $f$ using a truncated random tree Besov prior and write 
\begin{align}\label{eq:DefPrior}
\begin{split}
f(x) & =\sum_{(j,k)\in T} \langle f,\psi_{jk}\rangle\psi_{jk}(x)\\
& =\sum_{(j,k)\in \oT} \wt_{jk} g_{jk}\psi_{jk}(x),
\end{split}
\end{align}
where $g_{jk} \sim \Nn(0,\kappa^2)$ or $g_{jk} \sim \text{Laplace}(0,\kappa)$, and $\wt_{jk}\in\{0,1\}$ defines if a node $(j,k)\in \oT$ is included as described above. The variables $g_{jk}, t_{jk}$ are assumed to be mutually independent.

We consider noise of the form $W=\sum w_{jk}\psi_{jk}$ where $w_{jk}\sim \Nn(0,\sigma^2)$. 
The data can then be written as $M=\sum m_{jk}\psi_{jk}(x)$, where
\begin{align*}
m_{jk}=\left\{\begin{array}{ll}g_{jk}+w_{jk}, & \textrm{when}\ \wt_{jk}=1\\
w_{jk}, & \textrm{when}\ \wt_{jk}=0. \end{array} \right.
\end{align*}

Since $g_{jk}$ and $t_{jk}$ are assumed to be mutually independent and the hyperpriors for the $\beta_j$ are independent across levels, the posterior distribution of $\boldsymbol{g}=(g_{jk})_{(j,k)\in\oT}$, $\boldsymbol{t}=(t_{jk})_{(j,k)\in\oT}$, and $\boldsymbol{\beta}=(\beta_j)_{j=1}^J$, given the data $\boldsymbol{m}=(m_{jk})_{(j,k)\in\oT}$, can be written as
\begin{align*}
\pi(\boldsymbol{g},\boldsymbol{t},\boldsymbol{\beta}\ |\ \boldsymbol{m}) & \propto\pi(\boldsymbol{m}\ |\ \boldsymbol{g},\boldsymbol{t})\pi(\boldsymbol{g})\pi(\boldsymbol{t}\g \boldsymbol{\beta})\prod_{j=1}^J\pi_j(\beta_j)\\
& \propto\prod_{(j,k)\in \oT} \pi(m_{jk}\ |\ g_{jk},\wt_{jk})\pi(g_{jk})\pi(t_{jk}\g\beta_j)\Big(\frac{1}{2}-\beta_j\Big)^a\\
& = \prod_{(j,k)\in \oT} z_{jk}.
\end{align*}

When a node $(j,k)$ is included in the tree, we write
\begin{align*}
z_{jk}^1
& \propto \exp\Big(-\frac{1}{2\sigma^2}(m_{jk}-g_{jk})^2-R(g_{jk})\Big)\beta_j\Big(\frac{1}{2}-\beta_j\Big)^a,
\end{align*}
and when the node is excluded, we write 
\begin{align*}
z_{jk}^0
& \propto \exp\Big(-\frac{1}{2\sigma^2}m_{jk}^2-R(g_{jk})\Big)(1-\beta_j)\Big(\frac{1}{2}-\beta_j\Big)^a,
\end{align*}
where $\pi(g_{jk})\propto\exp(-R(g_{jk}))$. For the Gaussian base prior,
$R(g_{jk})=g_{jk}^2/(2\kappa^2)$, and for the Laplace base prior,
$R(g_{jk})=|g_{jk}|/\kappa$.

\subsection{Tree pruning with a fixed $\beta$}\label{Sec:PruningFixedBeta}

The case in which the wavelet density index $\beta$ is fixed and given was considered in \cite{Kekkonen2023}. We will discuss this algorithm shortly for completeness. For simplicity, we will describe the case $d=1$ below. 

If $\beta$ is fixed, we can write the optimisation problem in the form
\begin{align*}
\begin{split}
(\hat{\boldsymbol{g}},\hat{\boldsymbol{t}}) & = \argmax_{\boldsymbol{g},\boldsymbol{t}} \pi(\boldsymbol{g},\boldsymbol{t}\g \boldsymbol{m})\\
& = \argmax_{\boldsymbol{g},\boldsymbol{t}} \Bigg[ z^1_{00}
 +\sum_{k=0}^1 \bigg( t_{1k}\Big(z^1_{1k}
 + \!\!\! \sum_{(j',k')\prec(1,k)} \!\!\! z_{j'k'}\Big)
 + (1-t_{1k}) \!\!\! \sum_{(j',k')\preceq(1,k)} \!\!\! z^0_{j'k'}\bigg)\Bigg]
\end{split}
\end{align*}  
where we have assumed that the root node $(0,0)$ is always chosen. Above the optimisation problem is split into two subproblems concerning optimisation of the left and right subtrees with root nodes $(1,0)$ and $(1,1)$. A new root node $(1,k)$, $k=0,1$, is either included in the tree or excluded. Including the root node leads to two new optimisation problems, while excluding the node prunes the whole branch. Since we do not know the weight of the optimised subtree we cannot yet compare whether including the node or pruning the branch is an optimal choice. To solve the optimisation problem we continue the recursive splitting until the bottom level $j=J$, where $\wt_{Jk}=t_{Jk}$, is reached. At the bottom level we simply compare $z^1_{Jk}$ and $z^0_{Jk}$ componentwise. The optimal tree can then be built from bottom up by comparing the weight of an optimised subtree to the weight of pruning the branch. 

In \cite{Kekkonen2023} solving the MAP estimate for $\boldsymbol{g}$ and $\boldsymbol{t}$ with a fixed $\beta$ was shown to lead to a simple tree pruning algorithm when $g_{jk}\sim\Nn(0,1)$ and to tree pruning with additional soft thresholding when $g_{jk}\sim \text{Laplace}(0,\kappa)$. In tree pruning, all branches that do not carry enough information are pruned, resulting in zero wavelet coefficients in the sum \eqref{eq:DefPrior}. How much information is considered enough is determined by the wavelet density index $\beta$ which acts as a regularisation parameter. 

In the following we choose $\hat{\boldsymbol{\beta}}=(\hat{\beta}_j)_{j=1}^J$ recursively using levelwise MAPs. Starting from the bottom level $j=J$, we select $\hat{\beta}_j$, fix its value, and move up one level. The resulting procedure is a greedy bottom-up optimisation and does not in general give the joint MAP estimate of
$(\boldsymbol{g},\boldsymbol{t},\boldsymbol{\beta})$. 
A joint MAP estimate would require optimisation over a number of possible subtree configurations that grows exponentially with the size of the wavelet tree, together with the corresponding level-dependent parameters, making the joint optimisation computationally very expensive. The levelwise procedure provides a computationally efficient approximation while retaining the recursive structure of the tree-pruning algorithm.

Below we first consider the case $w_{jk}\sim \Nn(0,1)$ and $g_{jk}\sim\Nn(0,1)$. Proposition \ref{Prop:scaling} shows that in practice, under Gaussian base prior, the noise and prior variances $\sigma^2$ and $\kappa^2$ influence only the optimal choice of the wavelet density index $\beta$. It also provides a useful scaling relation that allows observations at different amplitudes and noise levels to be transformed to this normalised setting.

\subsection{Automatic $\boldsymbol{\beta}$ selection with Gaussian draws}\label{Sec:PruningGaussian}

For the levelwise MAP optimisation we work with the negative log posterior.
On the bottom level $j=J$ a node is included in the tree if 
\begin{align*}
-\log(z^1_{Jk})< -\log(z^0_{Jk}),        
\end{align*}
that is, if the weight of including the node is smaller than the weight of excluding it. If the two weights are equal, we exclude the node and favour the sparser model. Notice that the penalty for including a node $-\log(\beta)$, compared to the penalty for excluding a node $-\log(1-\beta)$, is regularising only when $\beta<0.5$.

We assume Gaussian base prior $g_{jk}\sim\Nn(0,1)$ and noise $w_{jk}\sim \Nn(0,1)$. On the bottom level the weight resulting from including a node is
\begin{align*}
Y^1_{Jk}
= \min_{g_{Jk}}\big(-\log(z^1_{Jk})\big)
= \frac{1}{4}m^2_{Jk}-\log\bigg(\beta_J\Big(\frac{1}{2}-\beta_J\Big)^a\bigg)
\end{align*}
which is attained by $\hat{g}^1_{Jk}=m_{Jk}/2$. If a node is not included we set $\hat{g}^0_{Jk}=0$ leading to weight
\begin{align*}
Y^0_{Jk}
= \min_{g_{Jk}}\big(-\log(z^0_{Jk})\big)
= \frac{1}{2}m^2_{Jk}-\log\bigg((1-\beta_J)\Big(\frac{1}{2}-\beta_J\Big)^a\bigg).
\end{align*}

We include a node in the tree if selecting it leads to a lower cost than pruning it, that is, $Y^1_{Jk}<Y^0_{Jk}$ or equivalently
\begin{align*}
\frac{1}{4}m^2_{Jk}-\log(\beta_J)
& < \frac{1}{2}m^2_{Jk}-\log(1-\beta_J). 
\end{align*}
Notice that the prior for $\beta$ does not play a role in whether a node is included or not. 
We see that, for a fixed value of $\beta_J$, the node $(J,k)$ should be included if 
\begin{align}\label{eq:GaussianChosen}
m_{Jk}^2> 4\log\bigg(\frac{1-\beta_J}{\beta_J}\bigg).
\end{align}

Denote by $A_\beta$ the set of included nodes, that is, the values of $k$ for which \eqref{eq:GaussianChosen} holds for a fixed $\beta$. We want to find the wavelet density index $\hat{\beta}_J$ that minimises the weight of the bottom row 
\begin{align}
\begin{split}
B_{J,A_{\beta_J}}
& = \sum_{k\in A_{\beta_J}} Y^1_{Jk}
+\sum_{k\notin A_{\beta_J}} Y^0_{Jk}\\
& = \sum_{k\in A_{\beta_J}}\frac{1}{4}m^2_{Jk} 
+ \sum_{k\notin A_{\beta_J}}\frac{1}{2}m^2_{Jk}\\
& \quad\quad +|A_{\beta_J}|\log\Big(\frac{1-\beta_J}{\beta_J}\Big)
-n_J\log\bigg((1-\beta_J)\Big(\frac{1}{2}-\beta_J\Big)^a\bigg)
\end{split}
\end{align}
where $|A_{\beta_J}|$ is the number of included nodes and $n_J=2^J$ is the number of nodes on the bottom row. We denote the data fit term by
$$L(m_J,\beta_J)=\sum_{k\in A_{\beta_J}}\frac{m^2_{Jk}}{4}+\sum_{k\notin A_{\beta_J}}\frac{m^2_{Jk}}{2}$$
and the regularisation penalty by
$$R(\beta_J)=|A_{\beta_J}|\log\Big(\frac{1-\beta_J}{\beta_J}\Big)
-n_J\log\bigg((1-\beta_J)\Big(\frac{1}{2}-\beta_J\Big)^a\bigg).$$
and write 
\begin{align}\label{eq:WeightBottomGaussian}
B_{J,A_{\beta_J}} = L(m_J,\beta_J)+R(\beta_J).
\end{align}

As $\beta_J$ increases, more nodes are included in the tree. The set of included nodes changes only at the threshold points
\begin{align*}
\beta_{Jk} = \frac{1}{1+\exp\big(\frac{1}{4}m^2_{Jk}\big)}.
\end{align*}
Since $\beta_{Jk}$ is a decreasing function of $m^2_{Jk}$, we reorder the nodes so that
\begin{align*}
m^2_{J(1)}\geq\cdots\geq m^2_{J(n_J)},
\end{align*}
which is equivalent to
\begin{align*}
\beta_{J(1)}\leq\cdots\leq\beta_{J(n_J)}.
\end{align*}
We set $\beta_{J(0)}=0$ and $\beta_{J(n_J+1)}=0.5$. Let $A_{J,r}$ denote the set of the $r$ nodes corresponding to the $r$ smallest threshold points. As $\beta_J$ increases, nodes are added to the selected set whenever a threshold point is crossed. With the tie convention above, the selected set is $A_{J,r}$ for $\beta_J\in\bigl(\beta_{J(r)},\beta_{J(r+1)}\bigr]$, $r=1,\dots,n_J$, and no nodes are selected for $\beta_J\in[0,\beta_{J(1)}]$.

For a fixed selected set $A_{J,r}$, the penalty $R(\beta_J)$, which we denote by $R_{J,r}(\beta_J)$, is strictly convex on $(0,0.5)$. Its minimiser over the admissible range $[0,0.5]$ is 
\begin{align}\label{eq:BetaInteriorMin}
\beta^*_{J,r}
=\frac{r}{q_{J,r}+\sqrt{q_{J,r}^2-2n_J(1+a)r}},
\qquad
q_{J,r}=a n_J+r+\frac{n_J}{2}.
\end{align} 
If $\beta^*_{J,r}\in[\beta_{J(r)},\beta_{J(r+1)}]$, it also minimises the penalty over this interval. If $\beta^*_{J,r}$ lies outside the interval, the minimum over the interval is attained at the nearest endpoint. Hence the constrained minimiser is
\begin{align}\label{eq:BetaIntervalMin}
\widetilde{\beta}_{J,r}
=\min\Big\{\beta_{J(r+1)},
\max\{\beta_{J(r)},\beta^*_{J,r}\}\Big\}.
\end{align}
At the endpoints $\beta_J=0$ and $0.5$, the cost is interpreted in the usual limiting sense. If a constrained minimiser coincides with a threshold point, we use the sparser of the two corresponding sets.

We now have a candidate minimiser $\widetilde{\beta}_{J,r}$ for each interval and choose the optimal $\beta_J$ as 
\begin{align*}
\hat\beta_J=\widetilde{\beta}_{J,\hat r_J}
\quad\text{where} \quad
\hat r_J
&=\argmin_{0\leq r\leq n_J}
B_{J,A_{J,r}}(\widetilde{\beta}_{J,r}).
\end{align*}

Once $\hat{\beta}_J$ has been selected we can use it to calculate the corresponding $g_{Jk}$ and $t_{Jk}$ as in \cite{Kekkonen2023}. After this we fix the values on the bottom row and move to the next level $j=J-1$. The process for the levels $j<J$ is similar to the bottom level but this time we compare the weight of an optimised subtree with root node $(j,k)$ to the weight of pruning the whole branch. 

Denote the weight of a previously optimised subtree with root node $(j,k)$ by $F_{jk}$. The weight of including the optimised subtree with a root node $(j,k)$ while paying the inclusion penalty is
\begin{align*}
Y^1_{jk}=F_{jk}-\log\bigg(\beta_j\Big(\frac{1}{2}-\beta_j\Big)^a\bigg),    
\end{align*}
while the weight of pruning the branch is 
\begin{align*}
Y^0_{jk} 
& = \frac{1}{2}\|m{|_{T_{jk}}}\|^2
-\log\bigg((1-\beta_j)\Big(\frac{1}{2}-\beta_j\Big)^a\bigg)
-\sum_{l=j+1}^J n_l \log\bigg((1-\hat{\beta}_l)\Big(\frac{1}{2}-\hat{\beta}_l\Big)^a\bigg)\\
& = \frac{1}{2}\|m{|_{T_{jk}}}\|^2
-\log\bigg((1-\beta_j)\Big(\frac{1}{2}-\beta_j\Big)^a\bigg)
-C_j,
\end{align*}
where $n_l=2^{l-j}$ is the number of the nodes in the subtree on the level $l$ and $\hat{\beta}_l$, $l=j+1,\dots,J$, are the previously optimised wavelet density indices. 

Notice that $C_j=\sum_{l=j+1}^J n_l \log((1-\hat{\beta}_l)(0.5-\hat{\beta}_l)^a)$ collects the penalties for not selecting a node on levels $l>j$ in the subtree $T_{jk}$. Since the corresponding wavelet density indices $\hat{\beta}_l$, $l>j$, have already been fixed in the bottom–up recursion, $C_j$ does not depend on $\beta_j$ and is a constant in the optimisation with respect to $\beta_j$.

As before we include the optimised subtree if it is cheaper than pruning the branch, that is, if
\begin{align*}
Y^1_{jk}< Y^0_{jk}.
\end{align*}
The corresponding threshold point is
\begin{align}\label{eq:HigherThresholds}
\beta_{jk}=\frac{1}{1+\exp(\frac{1}{2}\|m|_{T_{jk}}\|^2-F_{jk}-C_j)},
\end{align}
and the subtree with root node $(j,k)$ is included when $\beta_j>\beta_{jk}$. The threshold points depend on the data and the previously optimised subtree weights.

To find a levelwise optimal $\beta_j$ we want to minimise the weight the row adds to the tree. Namely, we want to find $\hat{\beta}_j$ that minimises
\begin{align*} 
B_{j,A_{\beta_j}}
& = \sum_{k\in A_{\beta_j}} Y^1_{jk}
+\sum_{k\notin A_{\beta_j}} Y^0_{jk}\\
& = \sum_{k\in A_{\beta_j}}F_{jk} 
+ \sum_{k\notin A_{\beta_j}}\Big(\frac{1}{2}\|m|_{T_{jk}}\|^2-C_j\Big)\\
& \quad\quad +|A_{\beta_j}|\log\Big(\frac{1-\beta_j}{\beta_j}\Big)
-n_j\log\bigg((1-\beta_j)\Big(\frac{1}{2}-\beta_j\Big)^a\bigg)\\
& = L(m_j,\beta_j)+R(\beta_j)
\end{align*}
where the data term is now 
$$L(m_j,\beta_j)= \sum_{k\in A_{\beta_j}}F_{jk} 
+ \sum_{k\notin A_{\beta_j}}\Big(\frac{1}{2}\|m|_{T_{jk}}\|^2-C_j\Big).$$
We absorb the penalty term $C_j$ into the data term so that the cost of pruning a branch can be compared directly with the optimised subtree weight $F_{jk}$, which accounts for the penalties associated with optimally including or excluding nodes within the subtree.

We order the threshold points \eqref{eq:HigherThresholds} from smallest to largest, denote them by $\beta_{j(1)},\dots,\beta_{j(n_j)}$, and set $\beta_{j(0)}=0$ and $\beta_{j(n_j+1)}=0.5$ as before. Denote by $A_{j,r}$ the set containing the $r$ nodes corresponding to the $r$ smallest threshold points. On each interval $(\beta_{j(r)},\beta_{j(r+1)}]$ the data term is constant. For a fixed $A_{j,r}$, the penalty term $R(\beta_j)$ is strictly convex. We therefore calculate $\beta^*_{j,r}$ and $\widetilde\beta_{j,r}$ using \eqref{eq:BetaInteriorMin} and \eqref{eq:BetaIntervalMin} with $J$ replaced by $j$, and choose
\begin{align*}
\hat\beta_j=\widetilde\beta_{j,\hat r_j}
\quad\text{where} \quad
\hat r_j
=\argmin_{0\leq r\leq n_j}B_{j,A_{j,r}}(\widetilde\beta_{j,r}).
\end{align*}
The algorithm is summarised in Algorithm \ref{alg:beta}. 

Because each $\hat{\beta}_j$ is fixed before moving to a coarser level, later pruning can remove some nodes that were selected on finer levels. A fine-level node may originally be retained even though it mainly represents noise, but its branch can still be pruned at a coarser level if the corresponding coefficients do not collectively contain enough information. 
In this case, changes in the finer-level value $\hat{\beta}_j$ that affect only subsequently pruned nodes would not alter the final reconstruction.

\begin{algorithm}[ht!]
\caption{Pseudo-code for automatic $\beta$ selection with Gaussian base prior}
\label{alg:beta}

Set $F_{Jk}=\frac14 m_{Jk}^2$, $k=0,\ldots,2^J-1$, and $C_J=0$\\

\For{$j=J$ \KwTo $1$}{

    \If{$j<J$}{
        Calculate
        \[
        C_j=\sum_{\ell=j+1}^J
        2^{\ell-j}
        \log\bigg((1-\hat\beta_\ell)
        \Big(\frac12-\hat\beta_\ell\Big)^a\bigg)
        \]
    }

    \For{$k=0$ \KwTo $2^j-1$}{
        Calculate
        \[
        \beta_{jk}
        =
        \bigg(
        1+\exp\Big(
        \frac12\|m|_{T_{jk}}\|^2-F_{jk}-C_j
        \Big)
        \bigg)^{-1}
        \]
    }

    Order the threshold points
    $0=\beta_{j(0)}\leq\beta_{j(1)}\leq\cdots\leq\beta_{j(n_j)}\leq \beta_{j(n_j+1)}=0.5$\\

    \For{$r=0$ \KwTo $n_j$}{
        Let $A_{j,r}$ contain the $r$ nodes with the smallest threshold points.\\
        Set
        \[
        q_{j,r}=a n_j+r+\frac{n_j}{2},
        \qquad
        \beta^*_{j,r}
        =
        \frac{r}{
        q_{j,r}
        +\sqrt{q_{j,r}^2-2n_j(1+a)r}}
        \]
        and
        \[
        \widetilde\beta_{j,r}
        =
        \min\Big\{
        \beta_{j(r+1)},
        \max\{\beta_{j(r)},\beta^*_{j,r}\}
        \Big\}.
        \]
        Calculate
        $B_{j,A_{j,r}}(\widetilde\beta_{j,r})$.
    }

    Set
    \[
    \hat r_j
    =
    \argmin_r
    B_{j,A_{j,r}}(\widetilde\beta_{j,r}),
    \qquad
    \hat\beta_j
    =
    \widetilde\beta_{j,\hat r_j}.
    \]

    Set
    \[
    t_{jk}
    =
    \mathbf{1}_{\{\hat\beta_j>\beta_{jk}\}},
    \qquad k=0,\ldots,2^j-1.
    \]

    \If{$j>1$}{
        \For{$k=0$ \KwTo $2^{j-1}-1$}{
            Set
            \[
            F_{j-1,k}
            =
            \frac14m_{j-1,k}^2
            +\min\{Y^1_{j,2k},Y^0_{j,2k}\}
            +\min\{Y^1_{j,2k+1},Y^0_{j,2k+1}\}.
            \]
        }
    }
}

Set
\[
\widetilde t_{00}=1,\qquad
\widetilde t_{jk}
=
\prod_{(j',k')\succeq(j,k)}t_{j'k'},
\qquad
g_{jk}=m_{jk}\widetilde t_{jk}.
\]
\end{algorithm}

\bigskip

\begin{proposition}\label{Prop:scaling}
Consider the denoising problem of estimating $f$ from a noisy measurement \eqref{eq:DenoisingMain}. We assume a truncated random tree Besov prior for $f$ as described in \eqref{eq:DefPrior}. Denote by $T^{1}_\beta(m_1)$ the denoised signal arising from tree pruning with a fixed $\beta$ when $w_i\sim \Nn(0,1)$ and $g_{jk}\sim\Nn(0,1)$. 

Then, for $\sigma,\kappa>0$, the denoised signal $T^{\kappa}_\beta(m_\sigma)$ arising from a measurement with noise $w_i\sim\Nn(0,\sigma^2)$ and prior choice $g_{jk}\sim\Nn(0,\kappa^2)$ is given by $\tilde{T}^1_{\tilde{\beta}}(m_1)$, where the coefficients $\hat{g}^1_{jk}=m_{jk}/2$ are replaced by $\hat{g}^1_{jk}=\kappa^2m_{jk}/(\kappa^2+\sigma^2)$ and
\begin{align*}
    \widetilde{\beta}=\frac{\beta^{c}}{\beta^{c}+(1-\beta)^{c}}
\end{align*}
with $c=\sigma^2(\kappa^2+\sigma^2)/(2\kappa^2)>0$.
\end{proposition}

Scaling the coefficients $\hat{g}^1_{jk}$ means scaling the denoised signal, which is undesirable. For this reason, in practice we set $\hat{g}^1_{jk}=m_{jk}$. Proposition \ref{Prop:scaling} then shows that once the optimal $\hat{\beta}^1_1$ has been estimated for $T^1_\beta(m_1)$ the corresponding optimal value $\hat{\beta}^\kappa_\sigma$ for $T^{\kappa}_\beta(m_\sigma)$ can be obtained directly from $\hat{\beta}^1_1$. 
Equivalently, one may view the noise and prior variances, 
$\sigma^2$ and $\kappa^2$, as influencing only the optimal choice of the regularisation parameter $\beta$. The same scaling relation also provides a principled way to rescale observations to the normalised setting used by the pruning algorithm, so that the likelihood and coefficient-prior terms are on an appropriate scale relative to the tree and hyperprior penalties. We use this rescaling in the image denoising experiments in Section \ref{Sec:Numerical}.

The following remark shows how scaling the input signal influences the resulting denoised signal.

\begin{remark}
We see from Proposition \ref{Prop:scaling} that, for a constant $c$, $T^1_\beta(cm_1) =cT^1_{\widetilde{\beta}}(m_1)$ where 
\begin{align*}
    \widetilde{\beta}=\frac{\beta^{c^2}}{\beta^{c^2}+(1-\beta)^{c^2}}.
\end{align*}
\end{remark}

\begin{proof}[Proof of Proposition \ref{Prop:scaling}]
We start by noting that when $\beta$ is fixed and $w_{jk}\sim\Nn(0,\sigma^2)$, $g_{jk}\sim\Nn(0,\kappa^2)$ the posterior of $g_{jk}\g \tilde{t}_{jk}=1$ is given by
\begin{align*}
z_{jk}^1 & \propto \exp\Big(-\frac{1}{2\sigma^2}(m_{jk}-g_{jk})^2
-\frac{1}{2\kappa^2}g_{jk}^2)\Big)\beta
\end{align*}
and the posterior for $g_{jk}\g \tilde{t}_{jk}=0$ is 
\begin{align*}
z_{jk}^0 & \propto \exp\Big(-\frac{1}{2\sigma^2}m_{jk}^2-\frac{1}{2\kappa^2}g_{jk}^2)\Big)(1-\beta).
\end{align*}

In the bottom row $j=J$ the weight resulting from including a node $(J,k)$ is simply
\begin{align*}
    Y_{Jk}^1=\min_{g_{Jk}}\Big(-\log(z_{Jk}^1)\Big)
=\frac{m_{Jk}^2}{2(\kappa^2+\sigma^2)}-\log(\beta),
\end{align*}
which is achieved at $\hat{g}_{Jk}=\kappa^2m_{Jk}/(\kappa^2+\sigma^2)$, and the weight for not including the node is 
\begin{align*}
    Y_{Jk}^0=\min_{g_{Jk}}\Big(-\log(z_{Jk}^0)\Big)
=\frac{m_{Jk}^2}{2\sigma^2}-\log(1-\beta),
\end{align*}
following from the choice $\hat{g}_{Jk}=0$.

A node $(J,k)$ is included in the subtree if including the node carries less weight than turning it off, that is, when $Y_{Jk}^1<Y_{Jk}^0$ or equivalently when 
\begin{align*}
    \log\Big(\frac{1-\beta}{\beta}\Big)
    < \frac{\kappa^2m_{Jk}^2}{2\sigma^2(\kappa^2+\sigma^2)}
    =\frac{1}{c}\frac{m_{Jk}^2}{4},
\end{align*}
where $c=\sigma^2(\kappa^2+\sigma^2)/(2\kappa^2)>0$. The corresponding threshold value is determined by
\begin{align*}
    \log\Bigg(\Big(\frac{1-\beta}{\beta}\Big)^c\Bigg)
    =\frac{m_{Jk}^2}{4},
\end{align*}
which corresponds to 
\begin{align*}
    \log\Big(\frac{1-\widetilde{\beta}}{\widetilde{\beta}}\Big)
    =\frac{m_{Jk}^2}{4},
\end{align*}
where $\widetilde{\beta}=\beta^c/(\beta^c+(1-\beta)^c)\in[0,0.5]$ for $\beta\in[0,0.5]$ and $c>0$. 
\end{proof}

\subsection{Automatic $\beta$ selection with Laplace draws}\label{Sec:PruningLaplace}

The levelwise optimisation of $\boldsymbol{\beta}$ when $g_{jk}\sim \text{Laplace}(0,\kappa)$ is similar to the Gaussian case though we need to take into account the additional soft thresholding which forces $\hat{g}_{jk}=0$ for all $|m_{jk}|\leq\frac{1}{\kappa}$. On the bottom level $j=J$ the weight of including a node is 
\begin{align*}
Y^1_{Jk} 
= \min_{g_{Jk}}\big(-\log(z^1_{Jk})\big)
=\begin{cases}
\frac{1}{2}m_{Jk}^2 -\log\Big(\beta_J\big(\frac{1}{2}-\beta_J\big)^a\Big) 
& \text{when } |m_{Jk}|\leq \frac{1}{\kappa} \\
\frac{1}{\kappa}|m_{Jk}|-\frac{1}{2\kappa^2}
-\log\Big(\beta_J\big(\frac{1}{2}-\beta_J\big)^a\Big)  
& \text{when } |m_{Jk}|> \frac{1}{\kappa},
\end{cases}
\end{align*}
which is attained at 
\begin{align*}
\hat{g}^1_{Jk}=
\argmin_{g_{Jk}}\big(-\log(z^1_{Jk})\big)
=\begin{cases}
      0 & \text{when}\ |m_{Jk}|\leq \frac{1}{\kappa} \\
     m_{Jk}-\text{sign}(m_{Jk})\frac{1}{\kappa} & \text{when}\ |m_{Jk}|> \frac{1}{\kappa}.
 \end{cases}
\end{align*}
As in the Gaussian case the weight of a pruned node is 
\begin{align*}
Y^0_{Jk} 
= \min_{g_{Jk}}\big(-\log(z^0_{Jk})\big)
=\frac{1}{2}m_{Jk}^2 -\log\Big((1-\beta_J)\big(\frac{1}{2}-\beta_J\big)^a\Big)
\end{align*}
which follows from setting $\hat{g}^0_{Jk}=0$.

A bottom row node is included in the tree if $Y^1_{Jk}<Y^0_{Jk}$. We immediately see that a node that carries little information $|m_{Jk}|\leq1/\kappa$ is always pruned on the bottom level since the penalty for including it is larger and there is no gain on the data fit. 

A bottom node is included in the tree if $|m_{Jk}|>1/\kappa$ and 
\begin{align}
\frac{1}{\kappa}|m_{Jk}|-\frac{1}{2\kappa^2}-\log(\beta_J)
< \frac{1}{2}m^2_{Jk}-\log(1-\beta_J).
\end{align}
In particular, given a measurement $|m_{Jk}|>1/\kappa$ the node $(J,k)$ is selected if 
\begin{align*}
\beta_J>\frac{1}{1+\exp\big(\frac{1}{2}m^2_{Jk}-\frac{1}{\kappa}|m_{Jk}|+\frac{1}{2\kappa^2}\big)}.    
\end{align*}

As before we want to solve 
\begin{align*}
\hat{\beta}_J = \argmin_{\beta_J\in[0,0.5]} B_{J,A_{\beta_J}}.
\end{align*}
The penalty term is the same as in \eqref{eq:WeightBottomGaussian} and the data term is
\begin{align*}
L(m_J,\beta_J) = \sum_{k\in A_{\beta_J}}\bigg(\frac{1}{\kappa}|m_{Jk}|-\frac{1}{2\kappa^2}\bigg)
+\sum_{k\notin A_{\beta_J}} \frac{1}{2}m^2_{Jk},
\end{align*}
where only nodes satisfying $|m_{Jk}|>1/\kappa$ can belong to $A_{\beta_J}$. The non-trivial threshold points are
\begin{align*}
\beta_{Jk}=\bigg(1+\exp\Big(\frac{1}{2}m^2_{Jk}-\frac{1}{\kappa}|m_{Jk}|+\frac{1}{2\kappa^2}\Big)\bigg)^{-1},
\qquad |m_{Jk}|>\frac{1}{\kappa}.
\end{align*}
Together with the endpoints $0$ and $0.5$, these points divide $[0,0.5]$ into intervals on which the selected set is fixed. We minimise the penalty on each interval as in \eqref{eq:BetaInteriorMin} and \eqref{eq:BetaIntervalMin}, and then compare the resulting values of the full row cost. 

Once $\hat{\beta}_J$ has been found we can move to the next level. When $j<J$ including a node into the tree can be beneficial even when $|m_{jk}|<1/\kappa$ if the optimised subtree starting from the node $(j,k)$ carries enough information. In this case we set $\hat{g}^1_{jk}=0$, $\tilde{t}_{jk}=1$, which provides additional regularisation within the optimised subtree.

Notice that if the fixed wavelet density parameter satisfies $\beta>0.5$, all nodes are included in the MAP tree. The tree constraint becomes inactive and the reconstruction reduces to coefficientwise soft thresholding. Therefore soft thresholding can be viewed as a MAP estimator arising from the standard Besov prior with independent Laplace coefficients. The tree-pruning formulation can be viewed as adding cross-scale structural sparsity to the usual Besov prior.

\section{Numerical Results \label{Sec:Numerical}}
We will now showcase the performance of the proposed automatic 
$\beta$-selection tree pruning method in 1D and 2D denoising, and 1D deconvolution.  
%The 2D experiments are conducted on open-access 2D images available at the Kodak Lossless True Color Image Suite. 

\subsection{Signal denoising and deconvolution}

We first consider one-dimensional signal denoising using the blocks test signal from \cite{Donoho1994}. Unless stated otherwise, the signal length is $2^{12}$. The signal is rescaled to have standard deviation $5$, and the observations are generated according to
\begin{align*}
    M_i = f(x_i)+W_i, \quad W_i\sim N(0,1), 
\end{align*}
that is, the noise standard deviation is one fifth of the standard deviation of the signal.
We report the relative error
\[
\mathrm{RE}(\widehat f,f)=100\,\frac{\|\widehat f-f\|_2}{\|f\|_2}.
\]
For the automatic wavelet density selection we use $a=0.1$ in the denoising experiments.

The original signal, a noisy measurement, and denoised reconstructions are shown in Figure \ref{HaarAndDaub2}. We can see that with a Gaussian base prior, Haar wavelets give a more accurate reconstruction of the overall piecewise constant structure of the signal, as expected. The Daubechies 2 wavelets nevertheless identify the jump locations well, despite the smoother and more spatially extended wavelet basis. The main loss of accuracy comes from the flat regions, where small oscillatory artefacts remain, particularly close to the jumps. We can also see that the selected tree is denser for Daubechies 2 wavelets, showing that more wavelet coefficients are needed to represent the same jumps. This illustrates that the tree structure can preserve the locations of discontinuities even when the chosen wavelet basis does not represent piecewise constant functions as sparsely as Haar wavelets. 

\begin{figure}[t]
\centering
\includegraphics[width=\textwidth]{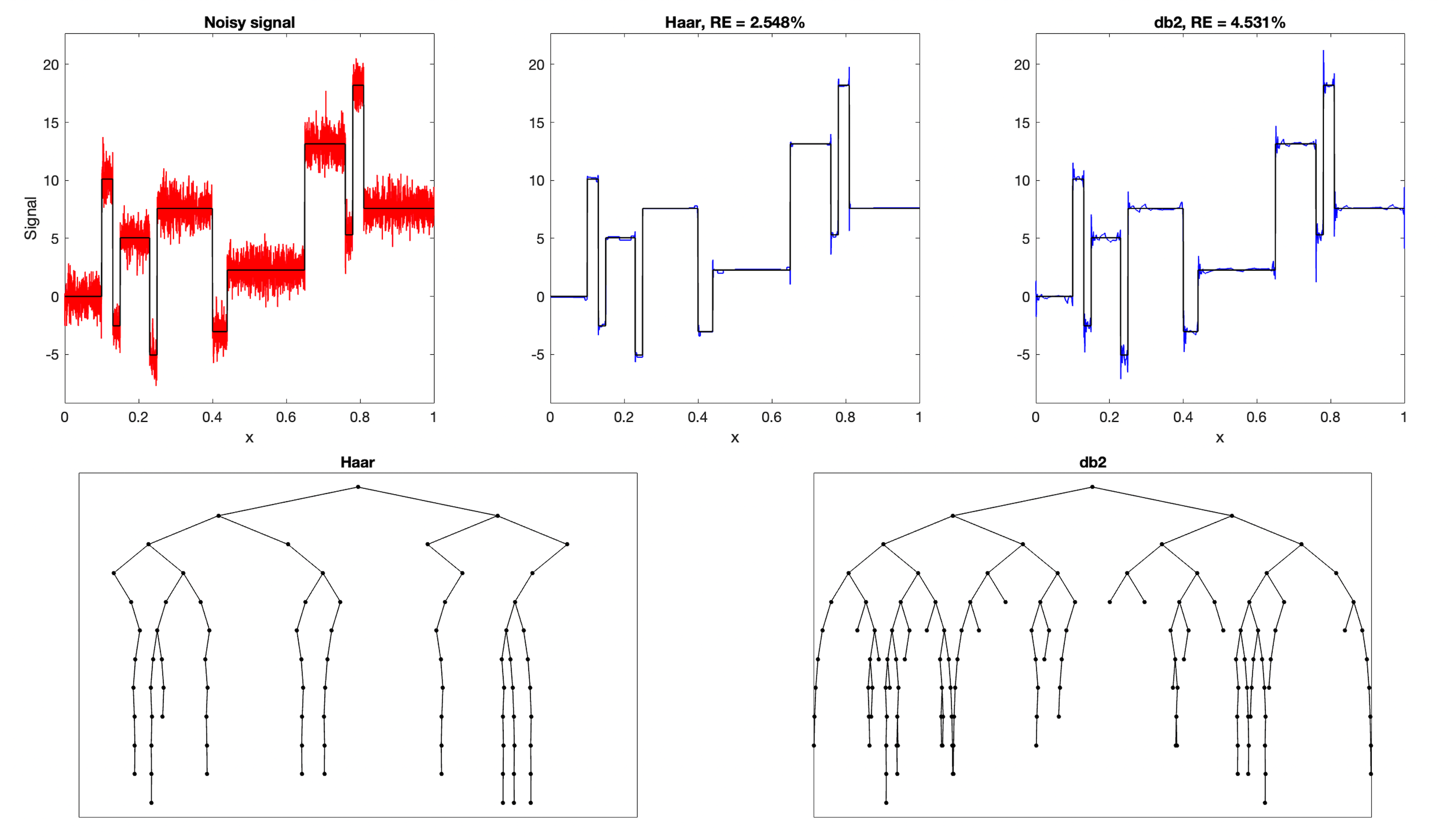}
    \caption{Top row: the true signal (black) and a noisy observation (red),
     and reconstructions using automatic $\boldsymbol{\beta}$ selection with Gaussian base prior and Haar (middle) and Daubechies 2 (right) wavelets. Bottom row: the corresponding selected wavelet trees.}
    \label{HaarAndDaub2}
\end{figure}

\begin{figure}[b]
\centering
\includegraphics[width=\textwidth]{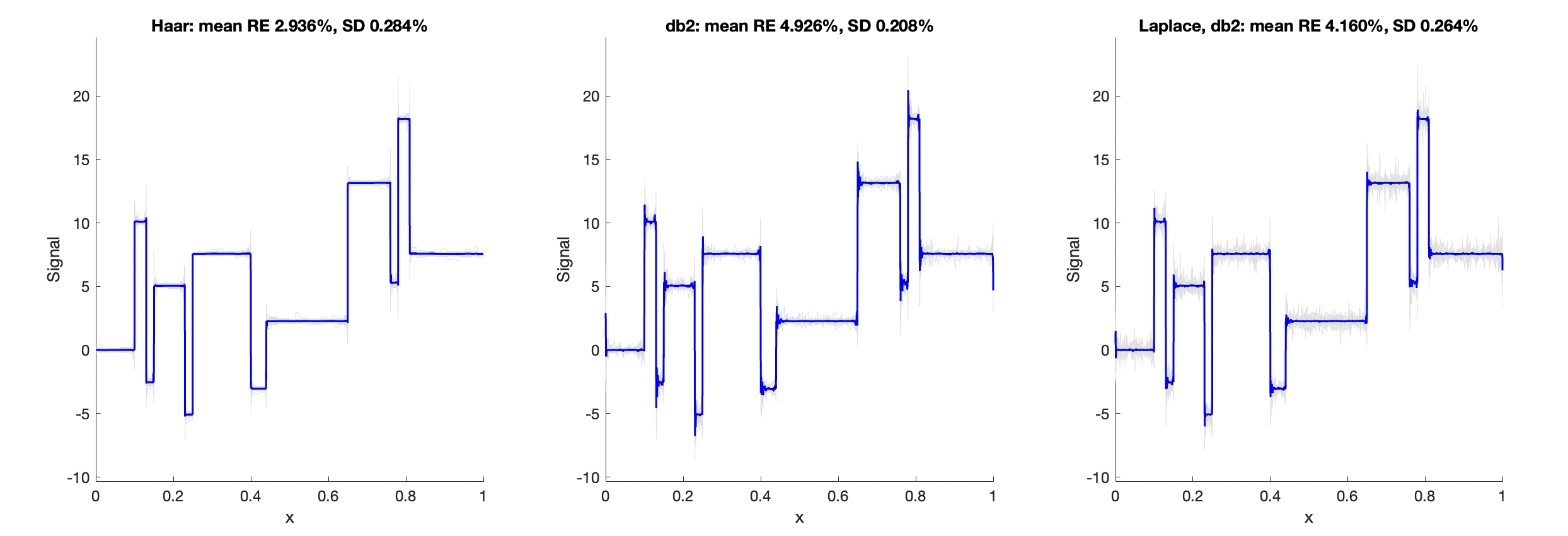}
    \caption{Denoising results over $50$ independent noise realisations. Individual reconstructions are shown in light grey and their pointwise mean in blue. Left: Gaussian base prior with Haar wavelets. Middle: Gaussian base prior with Daubechies 2 wavelets. Right: Laplace base prior with Daubechies 2 wavelets.}
    \label{50DrawsStudy}
\end{figure}

To study the variability with respect to the noise, we repeat the denoising experiment for $50$ independent noise realisations. Figure \ref{50DrawsStudy} shows all reconstructions in light grey and their pointwise mean in blue. With Gaussian base prior the Haar reconstruction has mean relative error $2.936\%$ with standard deviation $0.284\%$, while Daubechies 2 gives mean relative error $4.926\%$ with standard deviation $0.208\%$. With Laplace base prior and Daubechies 2 wavelets, resulting in soft pruning, the mean relative error drops to $4.160\%$ with standard deviation $0.264\%$. For the Laplace prior we use the soft-threshold level $1/\kappa=0.1$. On the same $50$ noise realisations, the Laplace Daubechies 2 reconstruction has a lower relative error than the Gaussian Daubechies 2 reconstruction in all cases. The choice of wavelet basis has a larger effect than the choice between the Gaussian and Laplace base priors in this example, although the additional soft thresholding improves the Daubechies 2 reconstruction consistently.

\begin{figure}[t]
\centering
\includegraphics[width=\textwidth]{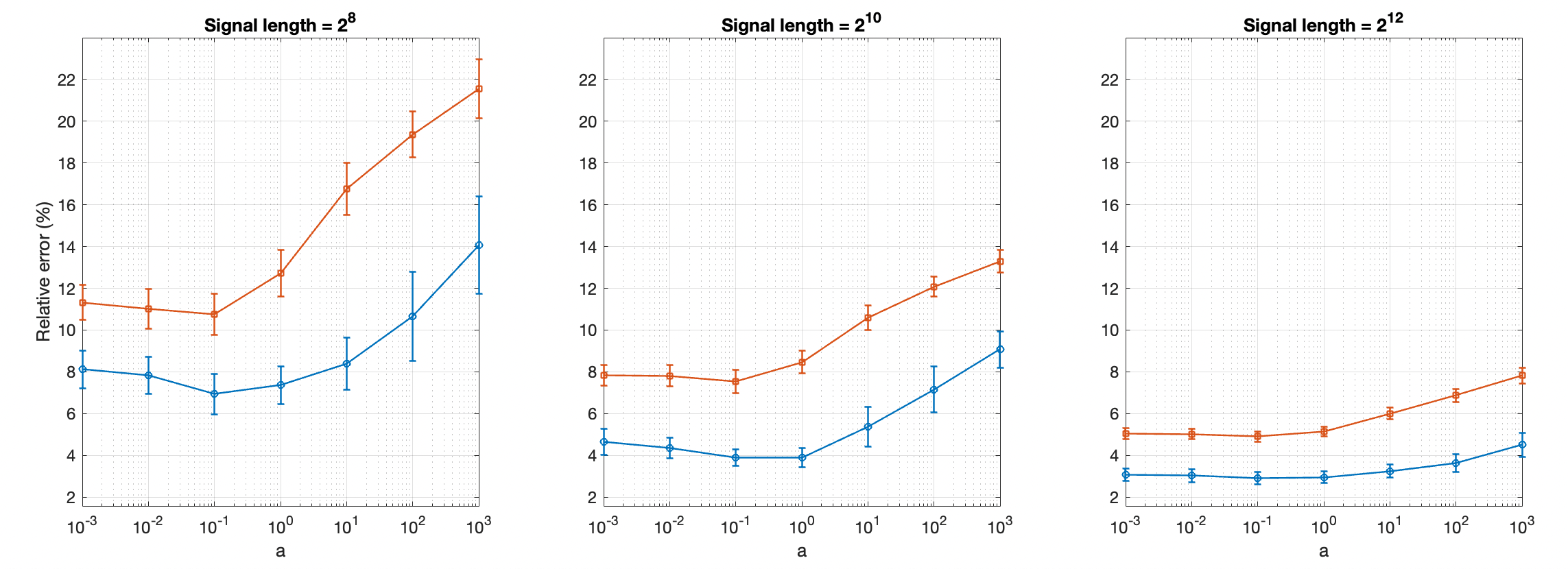}
    \caption{Sensitivity of the automatic $\boldsymbol{\beta}$ selection method to the hyperparameter $a$ and to the signal length. Each point is the mean relative error over $50$ independent noise realisations and the error bars show one standard deviation. Results are shown for Gaussian base prior with Haar (blue) and db2 (orange) wavelets.}
    \label{SensitivityOnA}
\end{figure}

\begin{figure}[b]
\centering
\includegraphics[width=\textwidth]{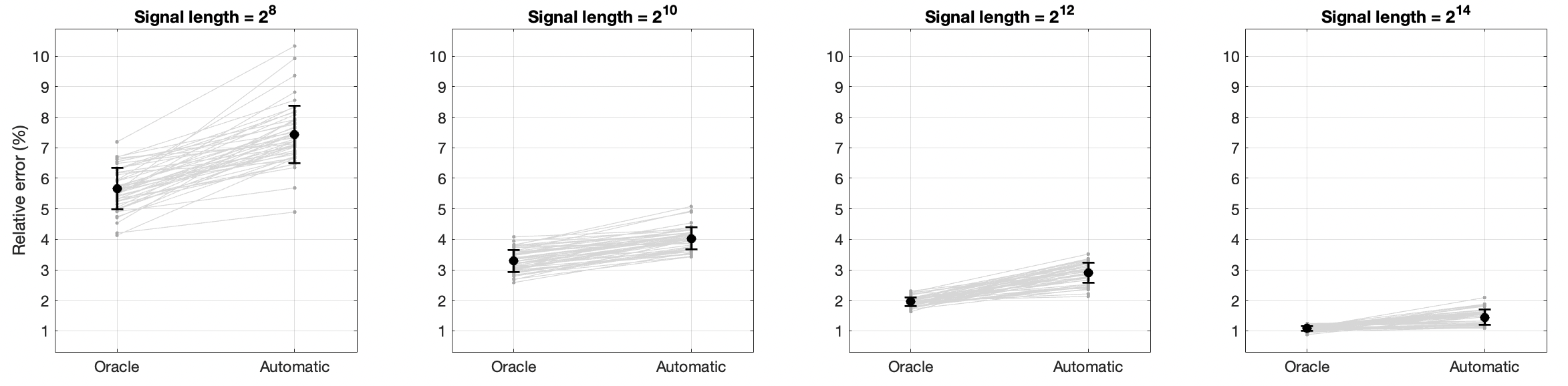}
    \caption{Comparison of automatic level-dependent $\boldsymbol{\beta}$ selection with an oracle-tuned fixed scalar $\beta$ for Haar wavelets. The comparison is based on $50$ independent noise realisations at each signal length. Error bars show one standard deviation.}
    \label{OracleVsAutomatic}
\end{figure}

Next we investigate the sensitivity to the hyperparameter $a$. Figure \ref{SensitivityOnA} shows the mean relative error over $50$ noise realisations for
\[
a\in\{10^{-3},10^{-2},10^{-1},1,10,10^2,10^3\}
\]
and signal lengths $2^8$, $2^{10}$, and $2^{12}$. The dependence on $a$ is qualitatively similar for Haar and Daubechies 2 wavelets. Increasing $a$ eventually leads to stronger pruning and larger reconstruction error. The dependence on $a$ also becomes less pronounced at finer discretisations. Based on this experiment we use $a=0.1$ as a representative value in the remaining denoising experiments.

We also compare the automatic level-dependent selection with an oracle-tuned fixed scalar $\beta$. For each noise realisation, the oracle value is chosen using the true signal to minimise the reconstruction error. This benchmark is therefore deliberately optimistic and is not available in practice. Figure \ref{OracleVsAutomatic} shows the comparison for Haar wavelets and signal lengths $2^8$, $2^{10}$, $2^{12}$, and $2^{14}$. The mean paired differences $\mathrm{RE}_{\mathrm{automatic}}-\mathrm{RE}_{\mathrm{oracle}}$
are $1.779$, $0.737$, $0.950$, and $0.366$ percentage points, respectively. The absolute gap is smaller at the finest resolution, although the dependence on resolution is not monotone. The comparison shows that especially on finer discretisation levels the automatic method gives similar reconstruction errors to the oracle without using the true signal to choose the regularisation parameter.

As a second one-dimensional example, we consider a mildly ill-posed deconvolution problem
\[
M=Af+W,
\]
where $A$ is a convolution operator and $f$ is a piecewise linear function shown in Figure \ref{StabilityDeconvolution}. The point spread function generating $A$ is proportional to $(0.05^2-x^2)^2$ on $[-0.05,0.05]$. The measurements contain Gaussian noise with standard deviation equal to $1\%$ of the maximum signal value.

\begin{figure}[b]
\centering
\includegraphics[width=\textwidth]{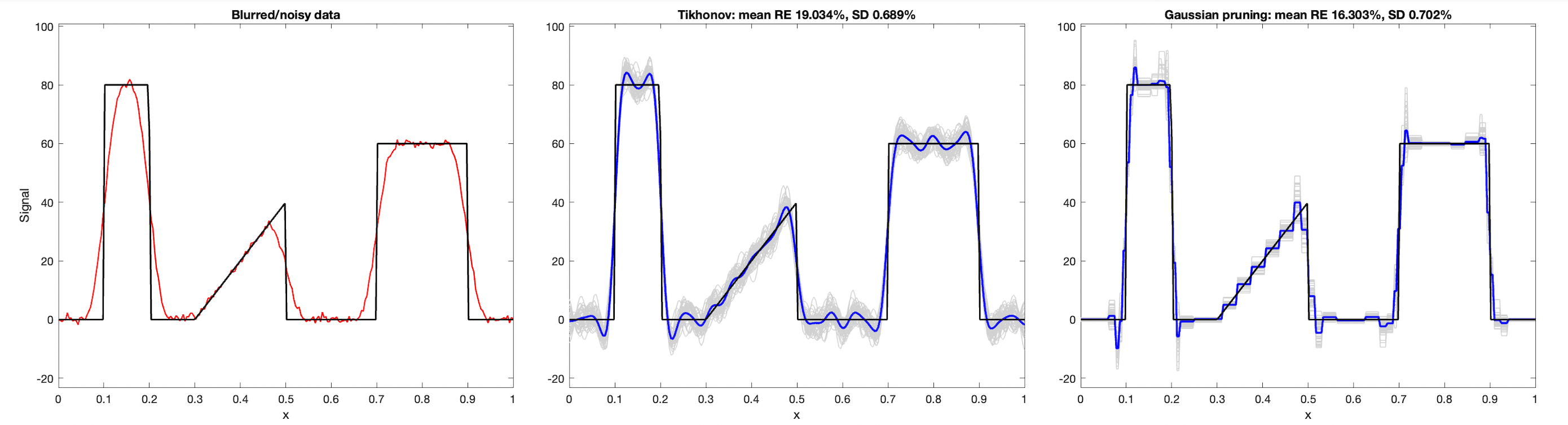}
    \caption{Deconvolution over $50$ independent noise realisations. Left: one blurred noisy measurement in red and the true signal in black. Middle: first-difference Tikhonov reconstructions. Right: Gaussian tree-pruning reconstructions with automatic $\boldsymbol{\beta}$ selection. Individual reconstructions are shown in light grey, their pointwise mean in blue, and the true signal in black.}
    \label{StabilityDeconvolution}
\end{figure}

We compare first-difference Tikhonov regularisation with a plug-and-play ADMM reconstruction using the Gaussian tree-pruning method. The Tikhonov parameter $\alpha=0.008$ was selected by minimising the reconstruction error on a preliminary noise realisation and is then kept fixed. For the plug-and-play method we initialise $x^0=v^0=u^0=0$ and iterate
\[
x^{k+1}
=\argmin_x\left\{\|Ax-M\|_2^2
+\rho\|x-(v^k-u^k)\|_2^2\right\},
\]
\[
v^{k+1}=T_{a}(x^{k+1}+u^k),
\qquad
u^{k+1}=u^k+x^{k+1}-v^{k+1},
\]
where $T_{a}$ denotes Gaussian tree pruning with Haar wavelets and automatic $\boldsymbol{\beta}$ selection. We use $\rho=1$ and $a=1$ throughout. The iteration is stopped if both
\[
\frac{\|x^{k+1}-v^{k+1}\|_2}{\|x^{k+1}\|_2}<10^{-4}
\quad\text{and}\quad
\frac{\|v^{k+1}-v^k\|_2}{\|v^k\|_2}<10^{-4},
\]
or a maximum of $1000$ iterations is reached. The same settings are used for all $50$ independent noise realisations.

Figure \ref{StabilityDeconvolution} shows one blurred noisy measurement together with the reconstructions arising from the $50$ noise realisations. First-difference Tikhonov gives mean relative error $19.034\%$ with standard deviation $0.689\%$. Gaussian tree pruning gives mean relative error $16.303\%$ with standard deviation $0.702\%$. The Haar representation produces visible staircasing on the linear part of the signal, but preserves the constant regions and discontinuities more accurately than the Tikhonov reconstruction.

\subsection{Image denoising}

The 2D experiments are conducted on three open-access images from the Kodak Lossless True Color Image Suite. The images are converted to greyscale and cropped to a square. All the images are $512\times512$ pixels. The resulting intensities are represented on the interval $[0,1]$. We use Daubechies 4 wavelets with periodic extension. The original images are shown in Figure \ref{OriginalImages}. 

For the numerical experiments, the horizontal, vertical and diagonal detail coefficients are organised into three separate quadtrees. Their tree structures are allowed to differ, while the level-dependent wavelet density parameter $\beta_j$ is shared between the three orientations. To estimate $\beta_j$, the inclusion/prune costs of all nodes at level $j$ from the three orientations are pooled, and the level-wise posterior objective is minimised jointly with respect to this single $\beta_j$. For each orientation there are $n_j=2^{2j}$ potential nodes on level $j$, so the pooled objective contains $3n_j$ node contributions; correspondingly, the 2D implementation uses the pooled count $3n_j$ in the level-wise hyperprior contribution. Once $\hat{\beta}_j$ has been selected, the three quadtrees are pruned separately using the common value. Thus $\hat{\beta}_j$ describes the overall sparsity at the level $j$, while the spatial and directional tree structures remain orientation-specific.
We use the hyperprior parameter $a=0.1$ throughout this section. 

\begin{figure}[b]
\centering
\includegraphics[width=\textwidth]{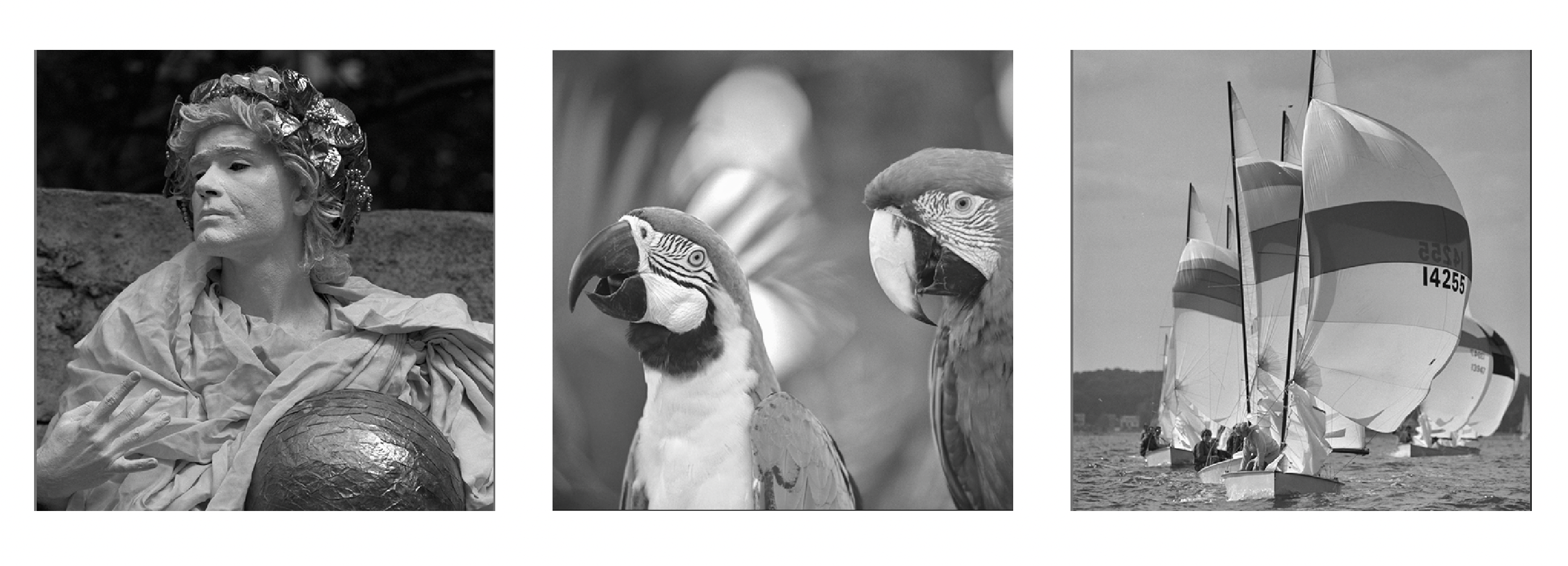}
    \caption{The three greyscale test images used in the 2D denoising experiments.}
    \label{OriginalImages}
\end{figure}

The noisy observations are generated according to
\[
M=f+\sigma Z,\qquad Z_{i}\sim N(0,1),
\]
with $\sigma\in\{0.20,0.15,0.10\}$. Pixel intensities outside $[0,1]$ have no practical meaning so after adding the Gaussian noise we clip the observations back to the admissible intensity range,
\[
M_i\leftarrow \min\{\max\{M_i,0\},1\}.
\]
This means that close to the ends of the intensity range, the image experiment is no longer exactly an additive Gaussian model. For noise realisations, the clipping affects $6.24$-$14.00\%$ of the pixels with $\sigma=0.20$, $2.78$-$10.49\%$ with $\sigma=0.15$, and $0.88$-$7.44\%$ with $\sigma=0.10$, depending on the image. Within each image the same standard-normal noise realisation is scaled by the three values of $\sigma$ in order to make the comparisons between noise levels more comparable. No post-processing rescaling of the reconstructed intensity range is used before evaluating the reconstructions.

The scale of the data is important for automatic tree-pruning. Applying the algorithm directly to images with intensities in $[0,1]$ would place the wavelet coefficients on a scale that is too small relative to the prior and hyperprior penalties, leading to excessive pruning. Therefore, a suitable rescaling of the data is required. Rather than choosing this scale empirically, we use the scaling derived in Proposition \ref{Prop:scaling} to motivate a deterministic normalisation. For the unclipped additive Gaussian model with base prior $g_{jk}\sim N(0,\kappa_G^2)$, Proposition \ref{Prop:scaling} gives the normalising factor
\[
q_G=
\sqrt{\frac{2\kappa_G^2}
{\sigma^2(\kappa_G^2+\sigma^2)}}.
\]
Because clipping means that the image observations above no longer follow the exact additive Gaussian model, the proposition does not give an exact probabilistic transformation for this experiment. We instead use $q_G$ as a model-motivated working normalisation: the Gaussian pruning algorithm is applied to $q_GM$ and the resulting reconstruction is divided by $q_G$. 
We note that $q_G\approx \sqrt{2/\sigma^2}$ when $\kappa_G$ is large relative to $\sigma$ and so the scaling depends only weakly on the precise value of \(\kappa_G\). We choose $\kappa_G=1$ for simplicity. For $\sigma=0.20,0.15,0.10$ we have $q_G=6.9,9.3,14.1$, respectively.
For the Laplace base prior, $g_{jk}\sim {\rm Laplace}(0,\kappa)$, we instead use the noise-normalised observations $M/\sigma$. In this normalisation the soft-threshold parameter is $\lambda=\frac{\sigma}{\kappa}.$

In the 1D experiments the observations contained unit-variance Gaussian noise, and the pruning algorithm was already formulated in this reference normalisation. Consequently no additional scaling was required. We also observed that moderate changes in the scaling factor had little effect on the resulting reconstructions.

To obtain a data-dependent estimate for $\lambda$ we use a simple plug-in estimate motivated by $\mathbb E|X|=\kappa$ for $X\sim{\rm Laplace}(0,\kappa)$. We first compute the Gaussian tree estimate and use the wavelet coefficients of the Gaussian tree to estimate $\kappa$. If $\widehat T_G$ denotes the selected Gaussian-tree coefficients, we set
\[
\widehat\kappa
=
\frac{1}{|\widehat T_G|}
\sum_{(j,k,\ell)\in\widehat T_G}|m_{jk}^{(\ell)}|
\qquad \text{and} \quad
\widehat\lambda=\frac{\sigma}{\widehat\kappa}.
\]

We compare four reconstruction methods: Gaussian tree pruning with automatic $\boldsymbol{\beta}$ selection; Laplace soft tree pruning with automatic $\boldsymbol{\beta}$ and the data-driven $\widehat\lambda$; Laplace soft tree pruning with automatic $\boldsymbol{\beta}$ and an oracle value of $\lambda$; and ordinary coefficientwise soft thresholding with an oracle threshold $\lambda$. For the two oracle methods the threshold is chosen to maximise the Structural Similarity Index (SSIM) against the original image. These oracle methods are benchmarks rather than practical data-driven procedures, since they use the ground truth to select the threshold.

Ordinary coefficientwise soft thresholding is included as a benchmark because it is a standard wavelet denoising method and corresponds to the MAP reconstruction under the standard Besov prior with independent Laplace coefficients. Comparing it with the Laplace tree-pruning method therefore illustrates the effect of introducing cross-scale tree structure into the prior.

Figure \ref{2DDenoisingResults} shows representative reconstructions: Image 1 with noise $\sigma=0.20$, Image 2 with noise $\sigma=0.15$, and Image 3 with noise $\sigma=0.10$. The noisy observations are shown in the left column, followed by Gaussian tree pruning with automatically chosen $\boldsymbol{\beta}$, Laplace tree pruning with automatically chosen $\boldsymbol{\beta}$ and $\lambda$, Laplace tree pruning with automatically chosen $\boldsymbol{\beta}$ and SSIM-oracle $\lambda$, and ordinary soft thresholding with an SSIM-oracle threshold.
We can see that the tree-based reconstructions retain visibly more fine-scale structure. Ordinary soft thresholding tends to produce an overly smooth reconstruction, losing local contrast and fine details, even though its threshold is chosen by an oracle. 

\begin{figure}[t]
\centering
\includegraphics[width=\textwidth]{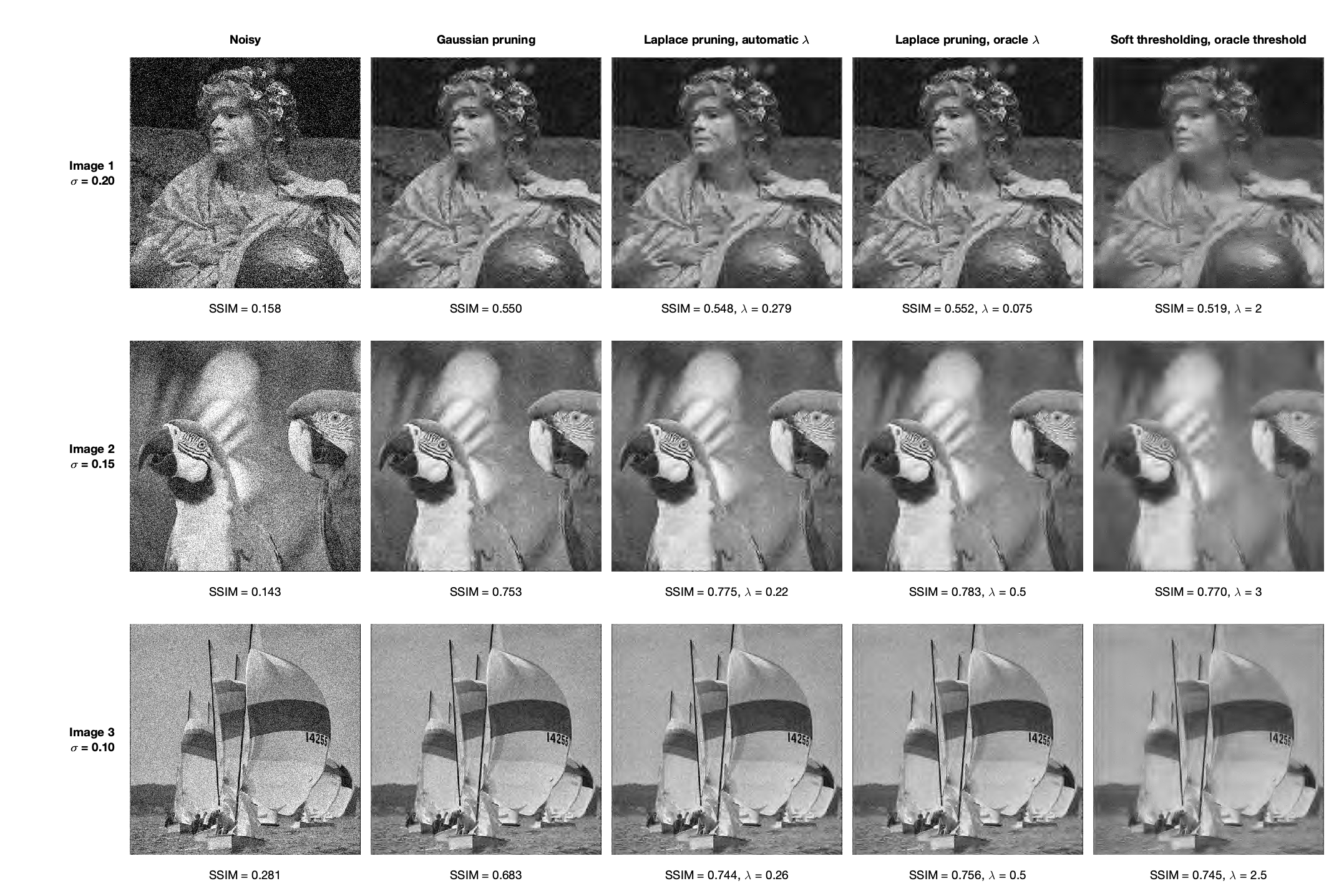}
    \caption{Representative 2D denoising results. The rows correspond to Image 1 with $\sigma=0.20$, Image 2 with $\sigma=0.15$, and Image 3 with $\sigma=0.10$. From left to right: noisy observation, Gaussian tree pruning with automatic $\boldsymbol{\beta}$, Laplace tree pruning with automatic $\boldsymbol{\beta}$ and automatic $\lambda$, Laplace tree pruning with automatic $\boldsymbol{\beta}$ and SSIM-oracle $\lambda$, and ordinary soft thresholding with an SSIM-oracle threshold.}
    \vspace{0.5cm}
    \label{2DDenoisingResults}
\end{figure}

\begin{table}[b]
\centering
\caption{SSIM and relative error for the 2D denoising experiments. The largest SSIM and smallest RE for each image and noise level are shown in bold.}
\label{tab:2Ddenoising}
\resizebox{\textwidth}{!}{%
\begin{tabular}{cc|cccc}
\hline
Image & $\sigma$ &
\shortstack{Gaussian tree\\SSIM/RE} &
\shortstack{Laplace automatic\\SSIM/RE} &
\shortstack{Laplace oracle\\SSIM/RE} &
\shortstack{Soft oracle\\SSIM/RE} \\
\hline
1 & 0.20 & 0.550/17.27 & 0.548/17.53 & \textbf{0.552}/\textbf{17.25} & 0.519/19.18 \\
1 & 0.15 & 0.611/14.83 & 0.611/15.01 & \textbf{0.614}/\textbf{14.82} & 0.581/16.37 \\
1 & 0.10 & 0.684/12.02 & 0.697/12.01 & \textbf{0.697}/\textbf{12.00} & 0.670/13.20 \\
\hline
2 & 0.20 & 0.737/\textbf{9.99} & 0.756/10.08 & \textbf{0.758}/10.45 & 0.749/12.66 \\
2 & 0.15 & 0.753/8.58 & 0.775/\textbf{8.52} & \textbf{0.783}/8.86 & 0.770/10.86 \\
2 & 0.10 & 0.778/6.79 & 0.808/\textbf{6.70} & \textbf{0.817}/7.23 & 0.807/8.11 \\
\hline
3 & 0.20 & 0.640/\textbf{10.40} & 0.668/10.46 & \textbf{0.671}/10.59 & 0.666/12.35 \\
3 & 0.15 & 0.655/9.15 & 0.696/\textbf{9.05} & \textbf{0.703}/9.30 & 0.696/10.72 \\
3 & 0.10 & 0.683/7.71 & 0.744/\textbf{7.47} & \textbf{0.756}/7.62 & 0.745/8.81 \\
\hline
\end{tabular}%
}
\end{table}

For quantitative comparison we report the SSIM and the relative error for all nine image and noise pairs in Table \ref{tab:2Ddenoising}.
The highest SSIM and the lowest relative error are highlighted for each image and noise level. The Laplace tree prior generally gives the highest SSIM, with the oracle-tuned version providing an upper benchmark. The data-driven Laplace reconstruction remains close to this benchmark and, in most of the experiments, improves substantially on Gaussian pruning, even though the selected $\widehat{\lambda}$ values are not necessarily close to the optimal $\lambda$. The relative-error results are less consistent with visual image quality and can differ significantly from the SSIM results. 

To assess sensitivity to the particular noise realisation, we repeated the complete comparison for $20$ independent Gaussian noise realisations for each of the three images and noise levels. Table \ref{tab:2Dstability} reports the mean and standard deviation of SSIM, together with the mean and standard deviation of the automatically selected $\widehat{\lambda}$. The data-driven methods are highly stable across the repeated experiments.

\begin{table}[t]
\centering
\caption{Mean $\pm$ standard deviation of SSIM over $20$ independent noise realisations. The final column reports the mean $\pm$ standard deviation of the automatically selected $\widehat{\lambda}$. The largest mean SSIM for each image and noise level is shown in bold.}
\label{tab:2Dstability}
\resizebox{\textwidth}{!}{%
\begin{tabular}{cc|cccc|c}
\hline
Image & $\sigma$ &
\shortstack{Gaussian tree\\SSIM} &
\shortstack{Laplace automatic\\SSIM} &
\shortstack{Laplace oracle\\SSIM} &
\shortstack{Soft oracle\\SSIM} &
\shortstack{Automatic\\$\widehat{\lambda}$} \\
\hline
1 & 0.20 & $0.5547\pm0.0020$ & $0.5517\pm0.0022$ & $\mathbf{0.5562\pm0.0020}$ & $0.5210\pm0.0016$ & $0.275\pm0.003$ \\
1 & 0.15 & $0.6118\pm0.0019$ & $0.6138\pm0.0019$ & $\mathbf{0.6158\pm0.0018}$ & $0.5839\pm0.0014$ & $0.267\pm0.002$ \\
1 & 0.10 & $0.6834\pm0.0012$ & $0.6989\pm0.0013$ & $\mathbf{0.6990\pm0.0013}$ & $0.6711\pm0.0008$ & $0.268\pm0.002$ \\
\hline
2 & 0.20 & $0.7381\pm0.0024$ & $0.7536\pm0.0023$ & $\mathbf{0.7576\pm0.0014}$ & $0.7479\pm0.0008$ & $0.218\pm0.004$ \\
2 & 0.15 & $0.7509\pm0.0031$ & $0.7751\pm0.0019$ & $\mathbf{0.7824\pm0.0016}$ & $0.7693\pm0.0008$ & $0.223\pm0.003$ \\
2 & 0.10 & $0.7751\pm0.0028$ & $0.8065\pm0.0016$ & $\mathbf{0.8175\pm0.0011}$ & $0.8064\pm0.0008$ & $0.209\pm0.002$ \\
\hline
3 & 0.20 & $0.6411\pm0.0040$ & $0.6669\pm0.0019$ & $\mathbf{0.6709\pm0.0022}$ & $0.6650\pm0.0010$ & $0.265\pm0.003$ \\
3 & 0.15 & $0.6561\pm0.0035$ & $0.6956\pm0.0023$ & $\mathbf{0.7034\pm0.0016}$ & $0.6949\pm0.0010$ & $0.259\pm0.003$ \\
3 & 0.10 & $0.6865\pm0.0040$ & $0.7440\pm0.0019$ & $\mathbf{0.7551\pm0.0016}$ & $0.7440\pm0.0005$ & $0.256\pm0.003$ \\
\hline
\end{tabular}%
}
\end{table}

For the automatically tuned Laplace method the standard deviation of SSIM is approximately $0.0012$--$0.0023$ across all nine image/noise combinations. For Images 2 and 3 it achieves a higher SSIM than Gaussian pruning in all $20$ repetitions at all three noise levels. For Image 1 it is slightly worse than Gaussian pruning for all repetitions at $\sigma=0.20$, better in $17$ of the $20$ repetitions at $\sigma=0.15$, and better in all $20$ repetitions at $\sigma=0.10$. The automatic threshold estimate is itself very stable: its standard deviation is below $0.005$ in every experiment. In contrast, the mean SSIM-oracle threshold ranges from $0.073$ to $0.538$ across the same experiments and shows noticeably greater variation both between images and noise levels.

For completeness, we also examined what happens if the oracle Laplace
threshold is selected by minimising relative error rather than maximising
SSIM. For Image 2, the mean RE-optimal values of $\lambda$ over $20$ noise
realisations were $0.033$, $0.086$, and $0.150$ for
$\sigma=0.20,0.15,0.10$, respectively, considerably smaller than the
corresponding SSIM-optimal values $0.470$, $0.525$, and $0.538$.
The mean SSIM values obtained with the RE-optimal thresholds were
$0.739$, $0.762$, and $0.800$, compared with $0.758$, $0.782$, and
$0.818$ for the SSIM-optimal thresholds. Since relative $\ell^2$ error can
favour reconstructions that are visibly less satisfactory, we use SSIM,
rather than relative error, to define the oracle thresholds in the image
experiments.

% \section{Conclusions}

% We introduced a hierarchical extension of random tree Besov priors in which the wavelet density parameter is allowed to vary across resolution levels and is selected automatically from the data. The resulting method combines level-dependent regularisation with the structured sparsity of the random tree prior. For both Gaussian and Laplace base priors, the parameters $\boldsymbol{\beta}$ can be selected efficiently together with the wavelet tree through a bottom-up sequence of level-wise MAP optimisations.

% We focused on a MAP-based reconstruction because the tree prior is designed to select a sparse multiscale support. This approach selects a single coherent subtree, while a posterior mean would average over different possible tree structures and may therefore smooth out local features such as discontinuities. 

% The numerical experiments indicate that the automatic procedure is stable with respect to noise realisations and in the one-dimensional experiments gave reconstruction errors close to those obtained using oracle-tuned fixed regularisation parameters. The experiments also illustrate the influence of the wavelet basis and the hyperparameter $a$, and show that the method can be incorporated into a plug-and-play scheme for an indirect deconvolution problem.

% Natural directions for future work include full Bayesian inference for the hierarchical tree model, theoretical analysis of the adaptive procedure, and applications to more general inverse problems, including PDE-based and nonlinear forward models.

 \newpage
%=====================================

\vspace{1cm}

\bibliographystyle{siam}
\bibliography{ref}

@Article{Castillo2017,
  author       = {Castillo, Isma{\"e}l},
  journal    = {Annales de l'Institut Henri Poincar{\'e}, Probabilit{\'e}s et Statistiques},
  title        = {P{\'o}lya tree posterior distributions on densities},
  number       = {4},
  organization = {Institut Henri Poincar{\'e}},
  pages        = {2074--2102},
  volume       = {53},
  year         = {2017},
}

@Article{Castillo2021,
  author       = {Castillo, Isma{\"e}l and Mismer, Romain},
  journal    = {Annales de l'Institut Henri Poincare (B) Probabilites et statistiques},
  title        = {Spike and slab {P}{\'o}lya tree posterior densities: Adaptive inference},
  number       = {3},
  organization = {Institut Henri Poincar{\'e}},
  pages        = {1521--1548},
  volume       = {57},
  year         = {2021},
}

@Article{Castillo2019,
  author  = {Castillo, Isma{\"e}l and Ro{\v{c}}kov{\'a}, Veronika},
  title   = {Uncertainty quantification for {B}ayesian {CART}},
  journal = {The Annals of Statistics},
  year    = {2021},
  volume  = {49},
  number  = {6},
  pages   = {3482--3509},
}

@article{Castillo2022,
  title={Optional {P}{\'o}lya trees: Posterior rates and uncertainty quantification},
  author={Castillo, Isma{\"e}l and Randrianarisoa, Thibault},
  journal={Electronic Journal of Statistics},
  volume={16},
  number={2},
  pages={6267--6312},
  year={2022},
  publisher={The Institute of Mathematical Statistics and the Bernoulli Society}
}

@article{Crouse1998,
  title={Wavelet-based statistical signal processing using hidden {M}arkov models},
  author={Crouse, Matthew S and Nowak, Robert D and Baraniuk, Richard G},
  journal={IEEE Transactions on signal processing},
  volume={46},
  number={4},
  pages={886--902},
  year={1998},
  publisher={IEEE}
}

@article{Romberg2001,
  title={Bayesian tree-structured image modeling using wavelet-domain hidden {M}arkov models},
  author={Romberg, Justin K and Choi, Hyeokho and Baraniuk, Richard G},
  journal={IEEE Transactions on image processing},
  volume={10},
  number={7},
  pages={1056--1068},
  year={2001},
  publisher={IEEE}
}

@article{Arridge2019,
  title={Solving inverse problems using data-driven models},
  author={Arridge, Simon and Maass, Peter and {\"O}ktem, Ozan and Sch{\"o}nlieb, Carola-Bibiane},
  journal={Acta numerica},
  volume={28},
  pages={1--174},
  year={2019},
  publisher={Cambridge University Press}
}

@article{Suuronen2020,
  title={Enhancing industrial {X}-ray tomography by data-centric statistical methods},
  author={Suuronen, Jarkko and Emzir, Muhammad and Lasanen, Sari and S{\"a}rkk{\"a}, Simo and Roininen, Lassi},
  journal={Data-Centric Engineering},
  volume={1},
  pages={e10},
  year={2020},
  publisher={Cambridge University Press}
}

@book {Daubechies1992,
    AUTHOR = {Daubechies, Ingrid},
     TITLE = {Ten lectures on wavelets},
    SERIES = {CBMS-NSF Regional Conference Series in Applied Mathematics},
    VOLUME = {61},
 PUBLISHER = {Society for Industrial and Applied Mathematics (SIAM),
              Philadelphia, PA},
      YEAR = {1992},
     PAGES = {xx+357},
      ISBN = {0-89871-274-2},
MRREVIEWER = {Albert\ Cohen},
       DOI = {10.1137/1.9781611970104},
       URL = {https://doi.org/10.1137/1.9781611970104},
}

@book {Cohen2003,
    AUTHOR = {Cohen, Albert},
     TITLE = {Numerical analysis of wavelet methods},
    SERIES = {Studies in Mathematics and its Applications},
    VOLUME = {32},
 PUBLISHER = {North-Holland Publishing Co., Amsterdam},
      YEAR = {2003},
     PAGES = {xviii+336}
}

@article{Wang2017,
  title={Bayesian inverse problems with $\ell_1$ priors: a randomize-then-optimize approach},
  author={Wang, Zheng and Bardsley, Johnathan M and Solonen, Antti and Cui, Tiangang and Marzouk, Youssef M},
  journal={SIAM Journal on Scientific Computing},
  volume={39},
  number={5},
  pages={S140--S166},
  year={2017},
  publisher={SIAM}
}

@article{Cui2021,
  title={Data-free likelihood-informed dimension reduction of {B}ayesian inverse problems},
  author={Cui, Tiangang and Zahm, Olivier},
  journal={Inverse Problems},
  volume={37},
  number={4},
  pages={045009},
  year={2021},
  publisher={IOP Publishing}
}

@article{Horst2025,
  title={Uncertainty Quantification for Linear Inverse Problems with {B}esov Prior: A Randomize-Then-Optimize Method},
  author={Horst, Andreas and Maboudi Afkham, Babak and Dong, Yiqiu and Lemvig, Jakob},
  journal={Statistics and Computing},
  volume={35},
  number={4},
  pages={101},
  year={2025},
  publisher={Springer}
}

@article{Schwab2024,
  title={Multilevel {M}onte {C}arlo {FEM} for elliptic {PDE}s with {B}esov random tree priors},
  author={Schwab, Christoph and Stein, Andreas},
  journal={Stochastics and Partial Differential Equations: Analysis and Computations},
  volume={12},
  number={3},
  pages={1574--1627},
  year={2024},
  publisher={Springer}
}

@article{Stein2023,
  title={Multilevel {M}arkov Chain {M}onte {C}arlo for {B}ayesian Elliptic Inverse Problems with {B}esov Random Tree Priors},
  author={Stein, Andreas and Hoang, Viet Ha},
  journal={arXiv preprint arXiv:2302.00678},
  year={2023}
}

@inproceedings{Adcock2017,
  title={Breaking the coherence barrier: A new theory for compressed sensing},
  author={Adcock, Ben and Hansen, Anders C and Poon, Clarice and Roman, Bogdan},
  booktitle={Forum of mathematics, sigma},
  volume={5},
  pages={e4},
  year={2017},
  organization={Cambridge University Press}
}

@article{Bui2015,
  title={A scalable algorithm for {MAP} estimators in {B}ayesian inverse problems with {B}esov priors},
  author={Bui-Thanh, Tan and Ghattas, Omar},
  journal={Inverse Problems and Imaging},
  volume={9},
  number={1},
  pages={27--53},
  year={2015},
  publisher={Inverse Problems and Imaging}
}

@Article{Donoho1994,
  Title                    = {Ideal spatial adaptation via wavelet shrinkage},
  Author                   = {Donoho, D. L. and Johnstone, I. M.},
  Journal                  = {Biometrika},
  Year                     = {1994},
  Number                   = {3},
  Pages                    = {425--455},
  Volume                   = {81}
}

@Article{Haario2004,
  Title                    = {Markov chain {M}onte {C}arlo methods for high dimensional inversion in remote sensing},
  Author                   = {Haario, H. and Laine, M. and Lehtinen, M. and Saksman, E. and Tamminen, J.},
  Journal                  = {Journal of the Royal Statistical Society: Series B (Statistical Methodology)},
  Year                     = {2004},
  Number                   = {3},
  Pages                    = {591--607},
  Volume                   = {66},

  Publisher                = {Wiley Online Library}
}

@article{Hamalainen2013,
  title={Sparse tomography},
  author={H\"am\"al\"ainen, Keijo and Kallonen, Aki and Kolehmainen, Ville and Lassas, Matti and Niinim\"aki, Kati and Siltanen, Samuli},
  journal={SIAM Journal on Scientific Computing},
  volume={35},
  number={3},
  pages={B644--B665},
  year={2013},
  publisher={SIAM}
}

@article{Kolehmainen2012,
  title={Sparsity-promoting {B}ayesian inversion},
  author={Kolehmainen, Ville and Lassas, Matti and Niinim{\"a}ki, Kati and Siltanen, Samuli},
  journal={Inverse Problems},
  volume={28},
  number={2},
  pages={025005},
  year={2012},
  publisher={IOP Publishing}
}

@Article{Lassas2004,
  Title                    = {Can one use total variation prior for edge-preserving {B}ayesian inversion?},
  Author                   = {Lassas, M. and Siltanen, S.},
  Journal                  = {Inverse Problems},
  Year                     = {2004},
  Pages                    = {1537--1564},
  Volume                   = {20},

  Publisher                = {IOP Publishing}
}

@Article{Lassas2009,
  Title                    = {Discretization-invariant {B}ayesian inversion and {B}esov space priors},
  Author                   = {Lassas, M. and Saksman, E. and Siltanen, S.},
  Journal                  = {Inverse Problems and Imaging},
  Year                     = {2009},
  Pages                    = {87-122},
  Volume                   = {3}
}

@Article{Rantala2006,
  Title                    = {Wavelet-based reconstruction for limited-angle {X}-Ray tomography},
  Author                   = {Rantala, M. and V\"ansk\"a, S. and J\"arvenp\"a\"a, S. and Kalke, M. and Lassas, M. and Moberg, J. and Siltanen, S.},
  Journal                  = {IEEE Transactions on Medical Imaging},
  Year                     = {2006},
  Number                   = {2},
  Pages                    = {210--217},
  Volume                   = {25},

  Publisher                = {IEEE}
}

@article{Reiss2008,
	Author = {Markus Rei{\ss}},
	Journal = {Annals of Statistics},
	Number = {4},
	Pages = {1957-1982},
	Title = {Asymptotic equivalence for nonparametric regression with multivariate and random design},
	Volume = {36},
	Year = {2008}}

@Article{Stuart2010,
  author =    {Stuart, Andrew M},
  title =     {Inverse problems: a {B}ayesian perspective},
  journal =   {Acta Numerica},
  year =      {2010},
  volume =    {19},
  pages =     {451--559},
  publisher = {Cambridge Univ Press}
}

@Book{Gine2015,
  author    = {Gin{\'e}, Evarist and Nickl, Richard},
  title     = {Mathematical foundations of infinite-dimensional statistical models},
  year      = {2015},
  volume    = {40},
  publisher = {Cambridge University Press},
}

@article{Kekkonen2023,
title = {Random tree {B}esov priors – Towards fractal imaging},
journal = {Inverse Problems and Imaging},
volume = {17},
number = {2},
pages = {507-531},
year = {2023},
issn = {1930-8337},
author = {Hanne Kekkonen and Matti Lassas and Eero Saksman and Samuli Siltanen}
}

@book{Tartar2007,
  title={An Introduction to {S}obolev Spaces and Interpolation Spaces},
  author={Luc Tartar},
  year={2007},
publisher={Springer}
}

@article{Dashti2012,
title = {Besov priors for {B}ayesian inverse problems},
journal = {Inverse Problems and Imaging},
volume = {6},
number = {2},
pages = {183-200},
year = {2012},
issn = {1930-8337},
author = {Masoumeh Dashti and Stephen Harris and Andrew Stuart},
}

@article{Agapiou2021,
  title={Rates of contraction of posterior distributions based on p-exponential priors},
journal = {Bernoulli},
pages = {1616–1642},
volume={27},
issue={3},
  author={Agapiou, Sergios and Dashti, Masoumeh and Helin, Tapio},
  year={2021}
}
\end{document}